\PassOptionsToPackage{dvipsnames}{xcolor}
\documentclass[11pt, oneside]{paper} 
\usepackage{amsmath, amsthm, amssymb,tikz-cd,geometry,quiver, xcolor,esint,mathtools,appendix}
\usepackage[dvipsnames]{xcolor}
\newtheorem{thm}{Theorem}[section]

\newtheorem{lem}[thm]{Lemma}
\newtheorem{cor}[thm]{Corollary}
\newtheorem{prop}[thm]{Proposition}

\theoremstyle{definition}
\newtheorem{defn}[thm]{Definition}
\newtheorem{rem}[thm]{Remark}
\newtheorem{eg}[thm]{Example}

\DeclareMathOperator{\ob}{ob}

\makeatletter
\providecommand*{\cupdot}{%
  \mathbin{%
    \mathpalette\@cupdot{}%
  }%
}
\newcommand*{\@cupdot}[2]{%
  \ooalign{%
    $\m@th#1\cup$\cr
    \sbox0{$#1\cup$}%
    \dimen@=\ht0 %
    \sbox0{$\m@th#1\cdot$}%
    \advance\dimen@ by -\ht0 %
    \dimen@=.5\dimen@
    \hidewidth\raise\dimen@\box0\hidewidth
  }%
}

\providecommand*{\bigcupdot}{%
  \mathop{%
    \vphantom{\bigcup}%
    \mathpalette\@bigcupdot{}%
  }%
}
\newcommand*{\@bigcupdot}[2]{%
  \ooalign{%
    $\m@th#1\bigcup$\cr
    \sbox0{$#1\bigcup$}%
    \dimen@=\ht0 %
    \advance\dimen@ by -\dp0 %
    \sbox0{\scalebox{2}{$\m@th#1\cdot$}}%
    \advance\dimen@ by -\ht0 %
    \dimen@=.5\dimen@
    \hidewidth\raise\dimen@\box0\hidewidth
  }%
}

\usepackage{tocloft}
\usepackage{hyperref}
\hypersetup{colorlinks,linkcolor=BlueViolet,citecolor=BlueViolet,urlcolor=BlueViolet, linktocpage}

\title{A categorical proof of the Carath\'eodory extension theorem}
\author{Ruben Van Belle \thanks{School of Mathematics, University of Edinburgh; ruben.van.belle@ed.ac.uk}}
\begin{document}
\maketitle
\begin{abstract}
The Carath\'eodory extension theorem is a fundamental result in measure theory. Often we do not know what a general measurable subset looks like. The Carath\'eodory extension theorem states that to define a measure we only need to assign values to subsets in a generating Boolean algebra.

To prove this result categorically, we represent (pre)measures and outer measures by
certain (co)lax and strict transformations. The Carath\'eodory extension then corresponds to
a Kan extension of strict transformations. We develop a general framework for extensions
of transformations between poset-valued functors and give several results on the existence
and construction of extensions of these transformations. We proceed by showing that transformations and functors corresponding to measures satisfy these results, which proves
the Carath\'eodory extension theorem.
\end{abstract}
\smalltableofcontents
\section{Introduction}



The Carath\'eodory extension theorem is an important result in measure theory. It guarantees the existence of the Lebesgue measure (or more generally the Lebesgue-Stieltjes measure) and of product measures. But it also is a key result in the proof of the Kolmogorov extensions theorem, which guarantees the existence of Brownian motion and which is closely related to martingale convergence results. 

If we want to define a measure on a measurable space $(X,\Sigma)$, we need to assign to every subset $A\in \Sigma$ an element of $[0,\infty]$. However, we often work with $\sigma$-algebras of the form $\sigma(\mathcal{B})$, where $\mathcal{B}$ is an algebra of subsets of $X$, i.e. closed under finite unions and complements. In this case it can be very difficult to know what a general measurable subset looks like and to assign real numbers to them in a $\sigma$-additive way. This problem is solved by the Carath\'eodory extension theorem. It states that a $\sigma$-additive map $\mathcal{B}\to [0,\infty]$ can be extended to a $\sigma$-additive map $\sigma(\mathcal{B})\to [0,\infty]$. Here, $\sigma$-additive should be interpreted as ‘$\sigma$-additive whenever  unions of countable pairwise disjoint collections exist’. Moreover, this extension is very often, but not always, unique. 
\[\begin{tikzcd}
	{\mathcal{B}} && {[0,\infty]} \\
	{\sigma(\mathcal{B})}
	\arrow[from=1-1, to=2-1]
	\arrow[from=1-1, to=1-3]
	\arrow[dashed, from=2-1, to=1-3]
\end{tikzcd}\]
The classical proof for this result is relatively long and technical (see for example Theorem 1.41 in \cite{klenke}).  It consists of different steps of extending and restricting back and requires several smart ‘tricks’ and constructions. In this paper we will give a categorical proof for the Carath\'eodory extension theorem using results on extensions of lax,colax and strict transformations between functors. Several parts of this proof look similar to steps in the classical proof. However, in our proof all constructions follow from the Kan extension formulas. Moreover, viewing the Carath\'eodory extension theorem in this categorical framework, allows us to compare it to extension results in other areas of mathematics. Furthermore, this technique allows us to easily generalize Carath\'eodory's result to measures taking values in other posets or spaces.

To do this we start by studying categories of certain transformations between functors and extensions of transformations along transformations. In section \ref{2} and \ref{3} we do this on a fairly abstract level. We discuss (co)laxification (co)monads and are in particular interested in the case that their (co)algebras are strict transformations. We give several abstract conditions for extensions of certain transformations to exist.

In section \ref{4} and \ref{5}, we give concrete constructions for the operations and extensions discussed in \ref{2} and \ref{3}. In this part (co)lax coends and (co)lax morphism classifiers play an important role. 

In the last two sections of the paper we apply the previous sections to measure theory. We start by representing inner and outer premeasures by certain transformations between certain functors. By applying the extension results to these functors, we obtain our categorical proof for Carath\'edory's extension result. We furthermore, make a distinction between \emph{left} and \emph{right} Carath\'edory extensions. The right one always exists and is the same as the extension in the classical proof, the left one however does not always exist. This is also related to the fact that inner measures and outer measures can behave surprisingly different from each other. Moreover, we immediately obtain a way to characterize these extensions by a universal property, namely as a maximal or minimal extension. This is interesting when uniqueness is not guaranteed. 

\begin{center}
\begin{tabular}{ c | c c }
  & Posets of transformations & Extensions of transformations \\ 
  \hline 
 Abstract theory & \ref{2} & \ref{3} \\  
 Concrete constructions & \ref{4} & \ref{5} \\
 Applications to measures & \ref{6} & \ref{7} \\
\end{tabular}
\end{center}

\textbf{Acknowledgments}: I would like to thank the two anonymous reviewers for their valuable comments and helpful suggestions.

\section{Posets of (co)lax transformations} \label{2}
In this section we will introduce and discuss posets of $\Sigma$-natural (co)lax transformations and strict transformations. In particular we will be interested in the embeddings between these and when these are (co)reflective. Furthermore, we will look at the (co)monads these adjunctions induce and study their (co)algebras. 

This section focuses on abstract existence results of (co)reflection operations; in Section \ref{4} we will give concrete constructions of these operations. In Section \ref{6}, we will represent inner (outer) premeasures as $\Sigma$-natural (co)lax transformations and strict transformations. The (co)reflection operations described in this section will then correspond to operations that turn inner (outer) premeasures into premeasures.

Let $\mathcal{C}$ be a small poset-enriched category. Let $\mathbf{Pos}$ be the category of posets and order-preserving maps, viewed as enriched over itself. Let $F$ and $G$ be enriched functors $\mathcal{C}\to \textbf{Pos}$. Let $\Sigma$ be a collection of morphisms in $\mathcal{C}$. 

We will study several generalizations of natural transformations. In the first variation we \emph{only} have naturality squares for morphisms in the fixed collection $\Sigma$.

\begin{defn}
    A $\Sigma$-\textbf{natural (general) transformation} $\tau: F\to G$ is a collection of order-preserving maps $(\tau_A:FA\to GA)_{A\in \text{ob}\mathcal{C}}$ such that \[\begin{tikzcd}
	FA & GA \\
	FB & GB
	\arrow["Ff"', from=1-1, to=2-1]
	\arrow["Gf", from=1-2, to=2-2]
	\arrow["{\tau_A}", from=1-1, to=1-2]
	\arrow["{\tau_B}"', from=2-1, to=2-2]
\end{tikzcd}\]
commutes for all $f\in \Sigma$. 
\end{defn}

In the other variations we will also allow \emph{weaker} naturality squares. 

\begin{defn}
    A $\Sigma$\textbf{-natural lax transformation} $\tau:F\to G$ is a collection of order-preserving maps $(\tau_A:FA\to GA)_{A\in \text{ob}\mathcal{C}}$ such that \[\begin{tikzcd}
	FA & GA \\
	FB & GB
	\arrow["Ff"', from=1-1, to=2-1]
	\arrow["Gf", from=1-2, to=2-2]
	\arrow["{\tau_A}", from=1-1, to=1-2]
	\arrow["{\tau_B}"', from=2-1, to=2-2]
	\arrow["\leq"{description}, draw=none, from=2-1, to=1-2]
\end{tikzcd}\]
for all morphisms $f:A\to B$ in $\mathcal{C}$ and such that this is an equality whenever $f\in \Sigma$.
\end{defn}

Dually, we can define  $\Sigma$-\textbf{natural colax transformations}, by reversing the inequality sign in the above definition. 

\begin{defn}
    A \textbf{strict transformation} $\tau:F\to G$ is a collection of order-preserving maps $(\tau_A:FA\to GA)_{A\in \text{ob}\mathcal{C}}$ such that \[\begin{tikzcd}
	FA & GA \\
	FB & GB
	\arrow["Ff"', from=1-1, to=2-1]
	\arrow["Gf", from=1-2, to=2-2]
	\arrow["{\tau_A}", from=1-1, to=1-2]
	\arrow["{\tau_B}"', from=2-1, to=2-2]
\end{tikzcd}\]
commutes for \emph{all} morphisms $f:A\to B$ in $\mathcal{C}$.
\end{defn}

Note that a strict transformation is the same as a general $\text{Mor}(\mathcal{C})$-natural transformation.  In the case that $\Sigma=\emptyset$, we will omit '$\Sigma$' in the terminology and notation. 

The set of $\Sigma$-natural transformations is partially ordered. The order is defined by \[ \tau^1\leq \tau^2 :\Leftrightarrow \tau^1_A(x)\leq \tau^2_A(x)\]  for all $A\in \ob\mathcal{C}$ and $x\in FA$ and for $\Sigma$-natural transformations $\tau^1$ and $\tau^2$.

The poset of $\Sigma$-natural transformations is denoted by $[F,G]^\Sigma$; the subposets of $\Sigma$-natural lax transformations, $\Sigma$-natural colax transformations and strict transformations are denoted by $[F,G]^\Sigma_l,[F,G]^\Sigma_c$ and $[F,G]_s$ respectively. We clearly have the following pullback square of inclusions.\[\begin{tikzcd}
	& {[F,G]_s} \\
	{[F,G]_c^\Sigma} && {[F,G]_l^\Sigma} \\
	& {[F,G]^\Sigma}
	\arrow[from=1-2, to=2-1]
	\arrow[from=1-2, to=2-3]
	\arrow["{U_c}"', from=2-1, to=3-2]
	\arrow["{U_l}", from=2-3, to=3-2]
	\arrow["\lrcorner"{anchor=center, pos=0.125, rotate=-45}, draw=none, from=1-2, to=3-2]
\end{tikzcd}\]

We will now give conditions on the functors $F$ and $G$ such that the inclusions in the above diagram are (co)reflective and such that these posets of transformations are complete. 

\begin{prop}\label{2.1}
Suppose $GA$ is cocomplete\footnote{Since $GA$ is a poset for every object $A$ in $\mathcal{C}$, $GA$ is also complete.} for all $A\in \mathrm{ob}\mathcal{C}$ and suppose $Gf$ preserves all joins for all $f\in \Sigma$. Then $[F,G]^\Sigma_l$ and $[F,G]^\Sigma$ are cocomplete and the inclusion $U_l:[F,G]^\Sigma_l\to [F,G]^\Sigma$ preserves all joins.
\end{prop}

\begin{proof}
Let $(\lambda^i)_{i\in I}$ be a collection of $\Sigma$-natural transformations. Define for all $A\in \ob\mathcal{C}$ and $x\in FA$, \[\lambda_A(x):=\bigvee_{i\in I}\lambda^i_A(x).\]
For $f:A\to B\in \Sigma$ and $x\in FA$, \[Gf(\lambda_A(x))=Gf\left(\bigvee_{i\in I}\lambda^i_A(x)\right)= \bigvee_{i\in I}Gf(\lambda^i_A(x))= \bigvee_{i\in I}\lambda^i_B(Ff(x))= \lambda_B(Ff(x)).\]
Therefore $\lambda$ is a $\Sigma$-natural transformation and it straightforward to verify that this is the join of $(\lambda^i)_{i\in I}$ in $[F,G]$. If $\lambda^i$ is lax for every $i\in I$, then for a map $f:A\to B$ in $\mathcal{C}$ and $x\in FA$, 
\[Gf(\lambda_A(x))=Gf\left(\bigvee_{i\in I}\lambda^i_A(x)\right)\leq \bigvee_{i\in I}Gf(\lambda^i_A(x))\leq \bigvee_{i\in I}\lambda^i_B(Ff(x))= \lambda_B(Ff(x)).\]
Here we used that $Gf$ is order-preserving and that $\lambda^i$ is lax for all $i\in I$. It follows that $\lambda$ is a $\Sigma$-natural lax transformation. This is the join of $(\lambda^i)_{i\in I}$ and clearly $U_l$ preserves this join.
\end{proof}
Using Proposition \ref{2.1}, we immediately obtain the following corollary.
\begin{cor}\label{2.2}
    If $GA$ is cocomplete for all $A\in \ob\mathcal{C}$ and if $Gf$ preserves joins for all $f\in \Sigma$, then $U_l:[F,G]^\Sigma_l\to [F,G]^\Sigma$ has a right adjoint ${\color{BrickRed}R_l:[F,G]^\Sigma\to [F,G]^\Sigma_l}$. 
\end{cor}

Since $U_l$ is full and faithful, the unit is an equality, i.e. $\lambda=R_lU_l\lambda$ for a $\Sigma$-natural lax transformation $\lambda$. Clearly, we also have dual results for Proposition \ref{2.1} and Corollary \ref{2.2}. In particular, if $GA$ is cocomplete for all $A\in \ob\mathcal{C}$ and $Gf$ preserves meets for all $f\in \Sigma$, then $U_c:[F,G]^\Sigma_c\to [F,G]^\Sigma$ has a left adjoint ${\color{OliveGreen}L_c:[F,G]^\Sigma\to[F,G]^\Sigma_c}$ and $\sigma=L_cU_c\sigma$ for a $\Sigma$-natural colax transformation $\sigma$.

Assume from now on that $GA$ is (co)complete and that $Gf$ preserves all joins and meets for $f\in \Sigma$. Using these adjunctions we obtain operations that turn lax transformations into colax transformations and vice versa. For a $\Sigma$-natural lax transformation $\lambda$, we denote \[\overline{\lambda}:=L_cU_l(\lambda),\] and we call $\overline{\lambda}$ the \textit{colaxification of}$\lambda$. For a $\Sigma$-natural colax transformation $\sigma$, we write \[\underline{\sigma}:=R_lU_c(\sigma),\] and we call $\underline{\sigma}$ the \textit{laxification of }$\sigma$.

In what follows we will often omit the forgetful functors. Using this convention we find that for a strict transformation $\tau$,  $$\underline{\tau}=\tau=\overline{\tau}.$$

The map $\overline{(-)}:[F,G]^\Sigma_l\to [F,G]^\Sigma_c$ is left adjoint to $\underline{(-)}:[F,G]^\Sigma_c\to [F,G]^\Sigma_l$. The following diagram summarizes this; 
\[\begin{tikzcd}
	{[F_,G]_l^\Sigma} && {[F,G]^\Sigma} && {[F,G]_c^\Sigma}
	\arrow[""{name=0, anchor=center, inner sep=0}, "{U_l}", shift left=2, from=1-1, to=1-3]
	\arrow[""{name=1, anchor=center, inner sep=0}, "{L_c}", shift left=2, color=OliveGreen, from=1-3, to=1-5]
	\arrow[""{name=2, anchor=center, inner sep=0}, "{U_c}", shift left=2, from=1-5, to=1-3]
	\arrow[""{name=3, anchor=center, inner sep=0}, "{R_l}", shift left=2, color=BrickRed, from=1-3, to=1-1]
	\arrow["{\overline{(-)}}"{description}, shift left=2, curve={height=-18pt}, from=1-1, to=1-5]
	\arrow["{\underline{(-)}}"{description}, shift left=5, curve={height=-12pt}, from=1-5, to=1-1]
	\arrow["\dashv"{anchor=center, rotate=-90}, draw=none, from=0, to=3]
	\arrow["\dashv"{anchor=center, rotate=-90}, draw=none, from=1, to=2]
\end{tikzcd}\]

The monad or closure operator on $[F,G]^\Sigma_l$ induced by this adjunction is denoted by $T$ and the induced comonad or interior operator on $[F,G]^\Sigma_c$ is denoted by $S$. 

The following proposition discusses the situation where (co)laxifications of $\Sigma$-natural co(lax) transformations are strict.
\begin{prop}\label{strictness}

The following are equivalent: 
\begin{enumerate}
    \item $\overline{\lambda}$ is strict for all $\lambda\in [F,G]^\Sigma_l$,
    \item $([F,G]^\Sigma_l)^T=[F,G]_s$,
    \item $\underline{\sigma}$ is strict for all $\sigma\in [F,G]^\Sigma_c$
    \item $([F,G]_c^\Sigma)^S=[F,G]_s$
\end{enumerate}
\end{prop}
\begin{proof}
$\underline{1\Rightarrow 2}$: For a strict transformation $\tau$, we have that $T\tau= \underline{(\overline{\tau})}=\underline{\tau}=\tau$, making it a $T$-algebra (or $T$-closed element). For a $T$-algebra (or $T$-closed element) $\lambda$, we have that $T\lambda=\lambda$. By the hypothesis, we also have that the laxifcation of $\overline{\lambda}$ is again $\overline{\lambda}$, i.e. $T\lambda =\overline{\lambda}$. Combining this gives that $\lambda=\overline{\lambda}$, showing that $\lambda$ is strict. 
$\underline{2\Rightarrow 3}$: Because $\underline{\sigma}=T\underline{\sigma}$, it follows that $\underline{\sigma}$ is a $T$-algebra and therefore strict, by the hypothesis.

$\underline{3\Rightarrow 4}$: This is similar to the proof of the implication $1\Rightarrow 2$. Every strict transformation $\tau$ is a coalgebra since $S\tau =\tau$. For an $S$-coalgebra (or $S$-open element) $\sigma$, we have $\underline{\sigma}=S\sigma = \sigma$. Therefore $\sigma$ is lax and thus strict. 

$\underline{4\Rightarrow 1}$: Similarly to the proof of the implication $2\Rightarrow 3$, it follows from the fact that $\overline{\lambda}=S\overline{\lambda}$. 
\end{proof}
We say that the triple $(F,G,\Sigma)$ satisfies the \textbf{strictness condition} if it satisfies the conditions in Proposition \ref{strictness}. In this case the inclusion $[F,G]_s\to [F,G]^\Sigma_l$ has a left adjoint ${\color{OliveGreen}L:[F,G]^\Sigma_l\to [F,G]_s}$ and that the inclusion $[F,G]_s\to [F,G]^\Sigma_c$ has a right adjoint  ${\color{BrickRed}R:[F,G]^\Sigma_c\to [F,G]_s}$. These operations turn $\Sigma$-natural (co)lax transformations into strict transformations in a universal way.

If the strictness condition is satisfied, then the poset of strict transformations is also complete and we can describe what joins and meets look like. This is explained in the following corollary.

\begin{cor}\label{2.4}
Suppose that $GA$ is complete for all $A\in \ob\mathcal{C}$ and that $Gf$ preserves joins for all $f\in \Sigma$. Suppose that $(F,G,\Sigma)$ satisfies the strictness condition. Let $(\tau^i)_{i\in I}$ be a collection in $[F,G]_s^\Sigma$. Then their join exists and is given by \[\bigvee_{i\in I} \tau^i=\overline{\lambda},\]
where $\lambda_A(x):=\bigvee_{i\in I}\tau^i_A(x)$ for all $A\in \ob\mathcal{C}$ and $x\in FA$.
\end{cor}

In the case that $(F,G,\Sigma)$ satisfies the strictness condition, we can summarize this section by the following diagram.
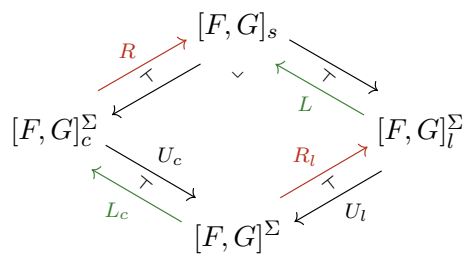
\[\begin{tikzcd}
	& {[F,G]_s} \\
	{[F,G]^\Sigma_c} && {[F,G]^\Sigma_l} \\
	& {[F,G]^\Sigma}
	\arrow[""{name=0, anchor=center, inner sep=0}, shift left=2, from=1-2, to=2-3]
	\arrow[""{name=1, anchor=center, inner sep=0}, "{U_l}", shift left=2, from=2-3, to=3-2]
	\arrow[""{name=2, anchor=center, inner sep=0}, shift left=2, from=1-2, to=2-1]
	\arrow[""{name=3, anchor=center, inner sep=0}, "{U_c}", shift left=2, from=2-1, to=3-2]
	\arrow["\lrcorner"{anchor=center, pos=0.125, rotate=-45}, draw=none, from=1-2, to=3-2]
	\arrow[""{name=4, anchor=center, inner sep=0}, "{R_l}", shift left=2, color=BrickRed, from=3-2, to=2-3]
	\arrow[""{name=5, anchor=center, inner sep=0}, "{L_c}", shift left=2, color=OliveGreen, from=3-2, to=2-1]
	\arrow[""{name=6, anchor=center, inner sep=0}, "L", shift left=2, color=OliveGreen, from=2-3, to=1-2]
	\arrow[""{name=7, anchor=center, inner sep=0}, "R", shift left=2, color=BrickRed, from=2-1, to=1-2]
	\arrow["\dashv"{anchor=center, rotate=51}, draw=none, from=5, to=3]
	\arrow["\dashv"{anchor=center, rotate=128}, draw=none, from=1, to=4]
	\arrow["\dashv"{anchor=center, rotate=51}, draw=none, from=6, to=0]
	\arrow["\dashv"{anchor=center, rotate=128}, draw=none, from=2, to=7]
\end{tikzcd}\]
\section{Extensions of transformations}\label{3}
In this section, we will focus on extensions of (co)lax and strict $\Sigma$-natural transformations. We will give results on their existence and their properties. In Section \ref{5}, we will give concrete constructions of these extensions in particular cases. We will apply this to extend premeasures to measures in Section \ref{7}, using the representation of inner and outer premeasures as $\Sigma$-natural transformations from Section \ref{6}. 

Let $\mathcal{C}$ be a small poset-enriched category and fix a collection $\Sigma$ of morphisms in $\mathcal{C}$. Let $F,G,H$ be enriched functors $\mathcal{C}\to \textbf{Pos}$ and assume that $GA$ is a complete for all $A\in\ob\mathcal{C}$ and that $Gf$ preserves joins and meets for all $f\in \Sigma$. Furthermore,  let $\iota:F\to H$ be a strict transformation.

Let $\lambda:F\to G$ be a $\Sigma$-natural general transformation (resp. lax, colax, strict). The \textbf{left extension of $\lambda$ along $\iota$} is a $\Sigma$-natural general transformation (resp. lax, colax, strict) $\mathrm{Lan}_\iota\lambda:H\to G$ such that $\textup{Lan}_\iota\lambda\circ \iota \geq \lambda$ and such that for every other $\lambda':H\to G$ with $\lambda'\circ \iota\geq \lambda$, we have $\lambda'\geq \mathrm{Lan}_\iota\lambda$;\footnote{The universal property determines the transformation, therefore we can talk about \emph{the} left extension.} 

\[\begin{tikzcd}
	F && G \\
	{} \\
	H
	\arrow[""{name=0, anchor=center, inner sep=0}, "\lambda"{description}, from=1-1, to=1-3]
	\arrow["\iota"{description}, from=1-1, to=3-1]
	\arrow[""{name=1, anchor=center, inner sep=0}, "{\mathrm{Lan}_\iota\lambda}"{description}, from=3-1, to=1-3]
	\arrow[""{name=2, anchor=center, inner sep=0}, "{\lambda'}"{description}, curve={height=30pt}, from=3-1, to=1-3]
	\arrow["\geq"{marking}, draw=none, from=3-1, to=0]
	\arrow["\leq"{marking}, draw=none, from=2, to=1]
\end{tikzcd}\]
We say that the extension is \textbf{proper} if $\mathrm{Lan}_\iota\lambda\circ \iota = \lambda$ and we call the extension \textbf{objectwise} if for all $A\in\ob\mathcal{C}$, \[(\mathrm{Lan}_\iota\lambda)_A=\mathrm{Lan}_{\iota_A}\lambda_A.\] Dually, we can define the \textbf{right extension of $\lambda$ along $\iota$}. 

If an extension is objectwise, then we can reduce everything to Kan extensions of order-preserving maps. This gives us for example the following lemma.
\begin{lem}\label{propernesslemma}
If the left extension of $\lambda$ along $\iota$ is objectwise and if $\iota_A$ is full and faithful for every $A\in \ob\mathcal{C}$, then the extension is proper.
\end{lem}
\begin{proof}
Because $\iota_A$ is full and faithful for all $A\in \ob\mathcal{C}$, we know by Corollary 6.3.9 in \cite{riehl} that $(\mathrm{Lan}_{\iota_A}\lambda_A)\circ \iota_A=\lambda_A$ for all $A \in \ob \mathcal{C}$. Because the extensions of $\lambda$ along $\iota$ is objectwise, we conclude that \[(\mathrm{Lan}_\iota\lambda)\circ \iota=\lambda.\] Therefore the extension is proper.
\end{proof}
Just as for Kan extensions of functors, extending transformations is adjoint to restricting transformations. The proof is essentially the same as the result for Kan extensions of functors (see for example Proposition 6.1.5 in \cite{riehl}).
\begin{lem}\label{lem}
If the restriction map $-\circ \iota: [H,G]_\bullet\to [F,G]_\bullet$ has a right (resp. left) adjoint, then the right (resp. left) extension of $\lambda$ along $\iota$ exists for all $\lambda \in [F,G]_\bullet$. Moreover, the right (resp. left) adjoint is given by $\mathrm{Ran}_\iota-$ (resp. $\mathrm{Lan}_\iota-)$.
\end{lem}

In the rest of this section we will look at conditions for extensions to exist and for them to be proper or objectwise. Extensions of $\Sigma$-natural general transformations always exist and are the best behaved.
\begin{prop}\label{3.3}
The right and left extension of $\tau$ along $\iota$ exists for every $\Sigma$-natural general transformation $\tau\in [F,G]^\Sigma$.
\end{prop}
\begin{proof}
It is clear by the proof of Proposition \ref{2.1} and its dual, that $[F,G]^\Sigma$ and $[H,G]^\Sigma$ are complete and that $-\circ \iota$ preserves all joins and meets. It follows now that the restriction map $-\circ \kappa$ has a left and a right adjoint. The claim now follows from Lemma \ref{lem}.\end{proof}
For $\Sigma$-natural lax transformations things become more difficult. However, we still have that right extensions exist.
\begin{prop}\label{3.4}
The right extension of $\lambda$ along $\iota$ exists for all $\lambda\in [F,G]_l^\Sigma$.
\end{prop}

\begin{proof}
Because $[F,G]_l^\Sigma$ is cocomplete by Proposition \ref{2.1} and because $-\circ \iota$ preserves all joins, the restriction map has a right adjoint by the Adjoint Functor Theorem for complete posets.
\end{proof}
\begin{rem}\label{3.5}
We have the following inequality \[\begin{tikzcd}
	{[H,G]_l^\Sigma} && {[F,G]_l^\Sigma} \\
	\\
	{[H,G]^\Sigma} && {[F,G]^\Sigma}
	\arrow["{\mathrm{Ran}_\iota-}"{description}, from=3-3, to=3-1]
	\arrow["{\mathrm{Ran}_\iota-}"{description}, from=1-3, to=1-1]
	\arrow["{U_l}"{description}, from=1-1, to=3-1]
	\arrow["{U_l}"{description}, from=1-3, to=3-3]
	\arrow["{\leq }"{description}, draw=none, from=1-1, to=3-3]
\end{tikzcd}\] If this is an equality and right extensions of $\Sigma$- general transformations are objectwise and proper, then so are right extensions of $\Sigma$-lax transformations.
\end{rem}
We have the following useful property about the existence of extensions of $\Sigma$-natural lax transformations and when they inherit properties from $\Sigma$-natural general transformations.
\begin{prop}\label{3.6}
We have the following inequality \[\begin{tikzcd}
	{[H,G]_l^\Sigma} && {[F,G]_l^\Sigma} \\
	\\
	{[H,G]^\Sigma} && {[F,G]^\Sigma}
	\arrow["{\leq }"{description}, draw=none, from=1-1, to=3-3]
	\arrow["{-\circ \iota}"{description}, from=1-1, to=1-3]
	\arrow["{R_l}"{description}, from=3-1, to=1-1]
	\arrow["{R_l}"{description}, from=3-3, to=1-3]
	\arrow["{-\circ \iota}"{description}, from=3-1, to=3-3]
\end{tikzcd}\]
If this is an equality, i.e. if $$R_l(-\circ \iota)=(R_l-)\circ \iota, $$then left extensions of $\Sigma$-natural lax transformations along $\iota$ exist and inherit objectwiseness and properness from left extensions of $\Sigma$-natural general transformations.

Moreover, right extensions of $\Sigma$-natural lax transformations inherit properness from right extensions of $\Sigma$-natural general transformations.
\end{prop}
\begin{proof}
The map $[H,G]^\Sigma\xrightarrow{-\circ \iota}[F,G]^\Sigma$ has a left adjoint by Proposition \ref{3.3} and the following diagram commutes \[\begin{tikzcd}
	{[H,G]_l^\Sigma} && {[F,G]_l^\Sigma} \\
	\\
	{[H,G]^\Sigma} && {[F,G]^\Sigma}
	\arrow["{U_l}"{description}, from=1-1, to=3-1]
	\arrow["{U_l}"{description}, from=1-3, to=3-3]
	\arrow["{-\circ \iota}"{description}, from=3-1, to=3-3]
	\arrow["{-\circ \iota}"{description}, from=1-1, to=1-3]
\end{tikzcd}\] Since $U_l$ is full and faithful, it follows from the adjoint lifting theorem in \cite{johnstone} that $[H,G]_l^\Sigma\xrightarrow{-\circ \iota} [F,G]_l^\Sigma$ has a left adjoint.  By the hypothesis we have that $R_l(-\circ \iota) = (R_l-)\circ \iota$. These are all right adjoint, so therefore their left adjoints commute as well, this means that the following diagram commutes \[\begin{tikzcd}
	{[H,G]_l^\Sigma} && {[F,G]_l^\Sigma} \\
	\\
	{[H,G]^\Sigma} && {[F,G]^\Sigma}
	\arrow["{\mathrm{Lan}_\iota-}"{description}, from=1-3, to=1-1]
	\arrow["{\mathrm{Lan}_\iota-}"{description}, from=3-3, to=3-1]
	\arrow["{U_l}"{description}, from=1-1, to=3-1]
	\arrow["{U_l}"{description}, from=1-3, to=3-3]
\end{tikzcd}\] 
It follows now easily that properness and objectwiseness are inherited from  left extensions of $\Sigma$-natural general transformations. 

Suppose now that right extensions of $\Sigma$-natural general transformations along $\iota$ are proper. Because $U_l(-\circ \iota) = (U_l-)\circ \iota$, their right adjoints commute as well. Together with the hypothesis, this leads to the following commutative square. \[\begin{tikzcd}
	{[F,G]_l^\Sigma} && {[H,G]_l^\Sigma} && {[F,G]_l^\Sigma} \\
	\\
	{[F,G]^\Sigma} && {[H,G]^\Sigma} && {[F,G]^\Sigma}
	\arrow["{\mathrm{Ran}_\iota-}"{description}, from=1-1, to=1-3]
	\arrow["{\mathrm{Ran}_\iota-}"{description}, from=3-1, to=3-3]
	\arrow["{-\circ \iota}"{description}, from=3-3, to=3-5]
	\arrow["{\circ \iota}"{description}, from=1-3, to=1-5]
	\arrow["{R_l}"{description}, from=3-1, to=1-1]
	\arrow["{R_l}"{description}, from=3-3, to=1-3]
	\arrow["{R_l}"{description}, from=3-5, to=1-5]
	\arrow["{1_{[F,G]}}"{description}, curve={height=18pt}, from=3-1, to=3-5]
\end{tikzcd}\]
We now have for a $\Sigma$-natural lax transformation $\lambda:F\to G$, \[\mathrm{Ran}_\iota \lambda \circ \iota = \mathrm{Ran}_\iota R_l(U_l)\lambda \circ \iota=R_lU_l\lambda = \lambda.\]
 
\end{proof}
We have dual results for Proposition \ref{3.4} and Proposition \ref{3.6} for $\Sigma$-natural colax transformation.

Under even more conditions, we can guarantee the existence of extensions of strict transformations. Let $T_1$ be the closure operator on $[H,G]_l^\Sigma$ as described before Proposition \ref{strictness} and let $T_2$ be the closure operator on $[F,G]_l^\Sigma$.
\begin{prop}\label{3.7}

Suppose that the following diagram commutes
\[\begin{tikzcd}
	{[H,G]_l^\Sigma} && {[F,G]_l^\Sigma} \\
	\\
	{[H,G]_l^\Sigma} && {[F,G]_l^\Sigma}
	\arrow["{-\circ \iota}", from=1-1, to=1-3]
	\arrow["{-\circ \iota}"', from=3-1, to=3-3]
	\arrow["{T_1}"{description}, from=1-1, to=3-1]
	\arrow["{T_2}"{description}, from=1-3, to=3-3]
\end{tikzcd}\]

Then $-\circ \iota$ induces an order-preserving map $([H,G]_l^\Sigma)^{T_1}\to ([F,G]_l^\Sigma)^{T_2}$ and this map has a right adjoint. 

Moreover, if right extensions of $\Sigma$-natural lax transformations along $\iota$ are proper (resp. objectwise), then so are right extensions of algebras along $\iota$.
\end{prop}
\begin{proof}
It follows from the hypothesis that the following square commutes. \[\begin{tikzcd}
	{([H,G]_l^\Sigma)^{T_1}} && {([F,G]_l^\Sigma)^{T_2}} \\
	\\
	{[H,G]_l^\Sigma} && {[F,G]_l^\Sigma}
	\arrow[from=3-1, to=1-1]
	\arrow[from=3-3, to=1-3]
	\arrow["{-\circ \iota}"{description}, from=3-1, to=3-3]
	\arrow["{-\circ \iota}"{description}, from=1-1, to=1-3]
\end{tikzcd}\]
By the adjoint lifting theorem in \cite{johnstone}, the right adjoint of $-\circ \iota: [H,G]_l^\Sigma\to [F,G]_l^\Sigma$ can be lifted to a right adjoint of the algebras. Because the left adjoints commute, so do the right adjoints. It follows that properness and objectwiseness are inherited from the extensions of $\Sigma$-natural lax transformations.
\end{proof}
Applying this to the case that the strictness condition holds, immediately gives us the following corollary. 
\begin{cor}\label{3.8}
Suppose that $(F,G,\Sigma)$ and $(H,G,\Sigma)$ satisfy the strictness condition and that $\overline{\lambda}\circ \iota = \overline{\lambda\circ \iota}$ for all $\lambda \in [H,G]_l$. Then, $[H,G]_s\xrightarrow{-\circ \iota}[F,G]_s$ has a right adjoint. Moreover, if right extensions of lax transformations along $\iota$ are proper (resp. objectwise), then so are right extensions of strict transformations along $\iota$. 
\end{cor}

Again, there are  dual results of Proposition \ref{3.7} and Corollary \ref{3.8} for left extensions of strict transformations. 


In general, the existence and properties of left and right extensions are not connected. We can for example have that right extensions don't exist, but left extensions do and are well-behaved. However, we do have the following connection between left and right extensions of strict transformations. 
\begin{prop}\label{3.9}
Suppose that $[H,G]_s\xrightarrow{-\circ \iota}[F,G]_s$ has a left and a right adjoint. Then, left extensions along $\iota$ are proper if and only if right extensions along $\iota$ are proper. 
\end{prop}
\begin{proof}
Let $\tau\in [F,G]_s$. By the universal property of extensions, we have $(\mathrm{Ran}_\iota \tau)_A \leq \mathrm{Ran}_{\iota_A}\tau_A$ and $\mathrm{Lan}_{\iota_A} \tau_A \leq (\mathrm{Lan}_{\iota}\tau)_A$ for all $A\in \ob\mathcal{C}$. We have the following inequalities for all $A\in \ob\mathcal{C}$: \[(\mathrm{Ran}_\iota \tau\circ \iota)_A \leq \mathrm{Ran}_{\iota_A}\tau_A\circ \iota_A\leq \tau_A\leq \mathrm{Lan}_{\iota_A}\tau_A\circ \iota_A\leq (\mathrm{Lan}_\iota \tau\circ \iota)_A\]
Suppose that left extensions along $\iota$ are proper, i.e. $\mathrm{Lan}_\iota\tau\circ \iota = \tau$. By the universal property of right extensions, it follows that \[\mathrm{Lan}_\iota\tau\leq \mathrm{Ran}_\iota\tau.\] Therefore $\tau=\mathrm{Lan}_\iota\tau\circ \iota = \mathrm{Ran}_\iota\tau\circ \iota$.

Similarly, if right extensions along $\iota$ are proper, we find by the universal property of left extensions that $\mathrm{Lan}_\iota\tau\leq \mathrm{Ran}_\iota\tau$ and we can conclude that $\mathrm{Lan}_\iota\tau\circ \iota = \mathrm{Ran}_\iota\tau\circ \iota=\tau$. 
\end{proof}

An overview of the maps we \emph{always} have is given by the following diagram. We assume that $(F,G,\Sigma)$ and $(H,G,\Sigma)$ satisfy the strictness condition. In the diagram, squares of the same colour commute and the dashed arrows indicate that these form the back of the parallelepiped. 

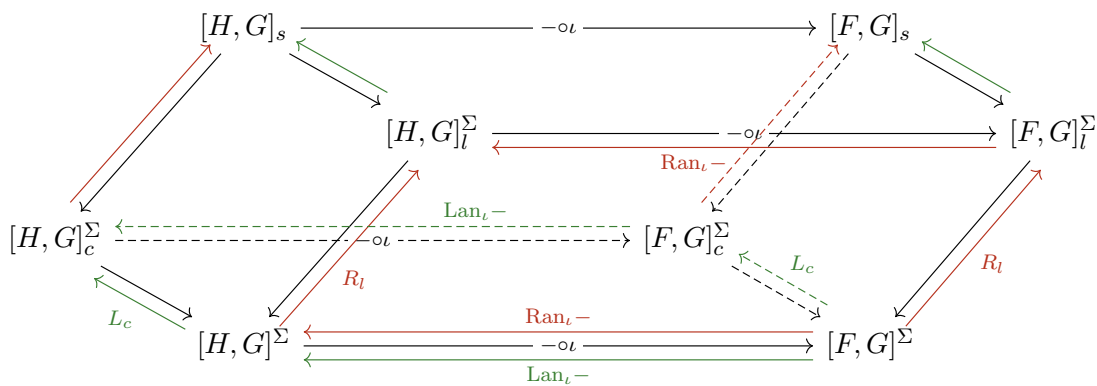
\[\begin{tikzcd}
	& {[H,G]_s} &&&& {[F,G]_s} \\
	&& {[H,G]_l^\Sigma} &&&& {[F,G]_l^\Sigma} \\
	{[H,G]_c^\Sigma} &&&& {[F,G]_c^\Sigma} \\
	& {[H,G]^\Sigma} & {} &&& {[F,G]^\Sigma}
	\arrow["{-\circ \iota}"{description}, from=4-2, to=4-6]
	\arrow["{-\circ \iota}"{description}, dashed, from=3-1, to=3-5]
	\arrow["{-\circ\iota}"{description}, from=2-3, to=2-7]
	\arrow["{-\circ\iota}"{description}, from=1-2, to=1-6]
	\arrow[from=1-2, to=3-1]
	\arrow[from=1-2, to=2-3]
	\arrow[from=3-1, to=4-2]
	\arrow[from=2-3, to=4-2]
	\arrow[dashed, from=1-6, to=3-5]
	\arrow[dashed, from=3-5, to=4-6]
	\arrow[from=1-6, to=2-7]
	\arrow[from=2-7, to=4-6]
	\arrow["{\mathrm{Lan}_\iota-}", shift left=2, color={OliveGreen}, from=4-6, to=4-2]
	\arrow["{\mathrm{Ran}_\iota-}"', shift right=2, color={BrickRed}, from=4-6, to=4-2]
	\arrow["{L_c}", shift left=2, color={OliveGreen}, from=4-2, to=3-1]
	\arrow[shift right=2, color={OliveGreen}, from=2-3, to=1-2]
	\arrow["{R_l}"'{pos=0.4}, shift right=2, color={BrickRed}, from=4-2, to=2-3]
	\arrow["{L_c}"', shift right=2, color={OliveGreen}, dashed, from=4-6, to=3-5]
	\arrow[shift right=2, color={OliveGreen}, from=2-7, to=1-6]
	\arrow["{R_l}"', shift right=2, color={BrickRed}, from=4-6, to=2-7]
	\arrow[shift left=2, color={BrickRed}, from=3-1, to=1-2]
	\arrow[shift left=2, color={BrickRed}, dashed, from=3-5, to=1-6]
	\arrow["{\text{Lan}_\iota-}"'{pos=0.3}, shift right=2, color={OliveGreen}, dashed, from=3-5, to=3-1]
	\arrow["{\mathrm{Ran}_\iota-}"{pos=0.6}, shift left=2, color={BrickRed}, from=2-7, to=2-3]
\end{tikzcd}\]

\section{The (co)laxification formula} \label{4}
In this section we will give concrete constructions for the operation $\overline{(-)}$ and $\underline{(-)}$ discussed in Section \ref{2}, in the case that the functors are well-behaved. Using these we can give explicit constructions of joins and meets in posets of $\Sigma$-natural transformations. These constructions will be used to construct (pre)measures from inner and outer (pre)measures and to describe joins and meets in posets of (pre)measures in Section \ref{6}.

Let $\mathcal{C}$ be a small poset-enriched category and let $\Sigma$ be a collection of morphisms in $\mathcal{C}$. Let $F$ and $G$ be enriched functors $\mathcal{C}\to \textbf{Pos}$ such that $GA$ is a complete for all $A\in \ob\mathcal{C}$ and such that $Gf$ preserves meets and joins for all $f\in \Sigma$. 

Intuitively it is clear that $(F,G,\emptyset)$ satisfying the strictness property is stronger than $(F,G,\Sigma)$ satisfying the strictness property. This is the content of the following results. This proposition motivates that it is enough to give constructions of these operations in the case that $\Sigma=\emptyset$.

\begin{prop}\label{4.1}
If $(F,G,\emptyset)$ satisfies the strictness condition, then $\overline{\lambda}^\Sigma=\overline{\lambda}^\emptyset$ for all $\lambda\in [F,G]^\Sigma_l$. In particular, $(F,G,\Sigma)$ satisfies the strictness condition.\footnote{Here $(-)^\Sigma:[F,G]_l^\Sigma\to [F,G]_c^\Sigma$ and $(-)^\emptyset:[F,G]_l^\emptyset\to [F,G]_c^\emptyset$ refer to the operations described in section \ref{2}.} 
\end{prop}
\begin{proof}
The inclusions $[F,G]_c^\Sigma\to [F,G]_c^\emptyset$ and $[F,G]^\Sigma\to [F,G]^\emptyset$ are meet-preserving maps between complete posets. Therefore, they have left adjoints ${\color{OliveGreen}{F_c^\Sigma}:[F,G]_c^\emptyset\to [F,G]_c^\Sigma}$ and ${\color{OliveGreen}{F^\Sigma}:[F,G]^\emptyset\to [F,G]^\Sigma}$ and we have that \[F_c^\Sigma(\sigma)=\sigma,\]  for every general $\Sigma$-natural colax transformation $\sigma:F\to G$. Moreover, the following square commutes, since their right adjoints do \[\begin{tikzcd}
	{[F,G]^\Sigma} && {[F,G]^\emptyset} \\
	\\
	{[F,G]_c^\Sigma} && {[F,G]_c^\emptyset}
	\arrow["{L^\Sigma_c}"', from=1-1, to=3-1]
	\arrow["{L^\emptyset_c}", from=1-3, to=3-3]
	\arrow["{F^\Sigma_c}", from=3-3, to=3-1]
	\arrow["{F^\emptyset}"', from=1-3, to=1-1]
\end{tikzcd}\]

This gives the following commutative diagram

\[\begin{tikzcd}
	{[F,G]_l^\Sigma} && {[F,G]_l^\emptyset} \\
	{[F,G]^\Sigma} && {[F,G]^\emptyset} \\
	{[F,G]_c^\Sigma} && {[F,G]_c^\emptyset}
	\arrow[ from=1-1, to=1-3]
	\arrow["{F^\Sigma}"{description}, from=2-3, to=2-1]
	\arrow["{F_c^\Sigma}", from=3-3, to=3-1]
	\arrow["{U_l^\Sigma}", from=1-1, to=2-1]
	\arrow["{L^\Sigma_c}", from=2-1, to=3-1]
	\arrow["{\overline{(-)}^\Sigma}"', curve={height=24pt}, from=1-1, to=3-1]
	\arrow["{U_l^\emptyset}"', from=1-3, to=2-3]
	\arrow["{L_c^\emptyset}"', from=2-3, to=3-3]
	\arrow["{\overline{(-)}^\emptyset}", curve={height=-18pt}, from=1-3, to=3-3]
\end{tikzcd}\]
	\\
For a $\Sigma$-natural lax transformation $\lambda:F\to G$, we know by the hypothesis that $\overline{\lambda}^\emptyset$ is strict and therefore it is a $\Sigma$-natural transformation. Therefore $\overline{\lambda}^\Sigma= F_c^\Sigma\overline{\lambda}^\emptyset=\overline{\lambda}^\emptyset$.
\end{proof}

For the rest of this section, let $\Sigma:=\emptyset$. We will write $[F,G]_\bullet$ to mean $[F,G]_\bullet^\emptyset$ and we will refer to $\emptyset$-natural general, lax and colax transformations as just \emph{general, lax and colax transformations} respectively. 

To obtain the constructions for $\overline{(-)}$ and $\underline{(-)}$, we will define new functors $F_\#, F_\#^\#$ and $F^\#$ and strict transformations \[\begin{tikzcd}
	& F \\
	{F_\#} && {F^\#} \\
	& {F_\#^\#}
	\arrow["{c}", from=3-2, to=2-1]
	\arrow["{l}"', from=3-2, to=2-3]
	\arrow["{c'}", from=2-1, to=1-2]
	\arrow["{l'}"', from=2-3, to=1-2]
\end{tikzcd}\]
such that $[F_\#,G]_s\cong [F,G]_c$; $[F_\#^\#,G]_s\cong [F,G]$ and $[F^\#,G]_s\cong [F,G]_l$. 

Using these isomorphisms, we can rewrite $\overline{(-)}:[F,G]_l\to [F,G]_c$ as \[[F,G]_l\cong [F^\#,G]_s\xrightarrow{-\circ l}[F^\#_\#,G]_s\xrightarrow{\mathrm{Lan}_{c}-}[F_\#,G]_s\cong [F,G]_c \]
and similarly, we can write $\underline{(-)}:[F,G]_c\to [F,G]_l$ as \[[F,G]_c\cong [F_\#,G]_s\xrightarrow{-\circ c}[F^\#_\#,G]_s\xrightarrow{\mathrm{Ran}_{l}-}[F^\#,G]_s\cong [F,G]_c .\]
If the extensions in these compositions are objectwise, we can give an explicit expression for $\overline{(-)}$ and $\underline{(-)}$. In this section we will discuss conditions for when this is the case.\newline

The enriched functor $F^\#:\mathcal{C}\to \textbf{Pos}$ is defined on objects by sending every object $A$ in $\mathcal{C}$ to the lax coend\footnote{Lax coends are explained in the Appendix \ref{A}} \[F^\#A:=\ointclockwise^{B\in \ob\mathcal{C}}\mathcal{C}(B,A)\times FB.\]
The functor $F^\#$ is called the \emph{lax morphism classifier} and have been discussed in \cite{blackwell,loregian,lack}.
The following proposition gives a concrete description of $F^\#A$ for $A\in \ob \mathcal{C}$.

We will now give an explicit description of this poset. For this, first define the following preorder $\mathcal{P}_A$, for an object $A$ in $\mathcal{C}$: 
\begin{itemize}
    \item The elements of $\mathcal{P}_A$ are pairs $(B\xrightarrow{g}A,y)$, where $g$ is a morphism in $\mathcal{C}$ and $y$ is an element of $FB$,
    \item We write $(B_1\xrightarrow{g_1}A,y_1) \preceq $ $(B_2\xrightarrow{g_2}A,y_2)$ if there exists a morphism $s:B_2\to B_1$ in $\mathcal{C}$ such that $g_1s\leq g_2$ and $y_1\leq Fs(y_2)$. 
\end{itemize}

\begin{prop}\label{4.2}
Let $A$ be an object of $\mathcal{C}$. Then $F^\#A$ is the poset induced\footnote{A preorder $P$ induces a poset by identifying elements $a$ and $b$ in $P$ with each other if $a\leq b$ and $b\leq a$.} by the preorder $\mathcal{P}_A$.
\end{prop}

\begin{proof}
Let $\hat{\mathcal{P}}_A$ denote the poset induced by the preorder $\mathcal{P}_A$.  We will now show that $\hat{\mathcal{P}}_A$ satisfies the universal property of lax coends. 

For $B\in \ob\mathcal{C}$, there clearly is an order-preserving map $e_B:\mathcal{C}(B,A)\times FB\to \hat{\mathcal{P}}_A$. A map $s:B_2\to B_1$ induces 
\[\begin{tikzcd}
	{\mathcal{C}(B_1,A)\times FB_2} && {\mathcal{C}(B_2,A)\times FB_2} \\
	\\
	{\mathcal{C}(B_1,A)\times FB_1} && {\hat{\mathcal{P}}_A}
	\arrow["{\text{Id}\times Fs}"{description}, from=1-1, to=3-1]
	\arrow["{\mathcal{C}(s,A)\times \text{Id}}", from=1-1, to=1-3]
	\arrow["{e_{B_1}}"', from=3-1, to=3-3]
	\arrow["{e_{B_2}}", from=1-3, to=3-3]
	\arrow["\preceq"{description}, draw=none, from=3-1, to=1-3]
\end{tikzcd}\]

Indeed, for $f:B_1\to A$ and $y\in FB_2$, we have that $$(f,Fs(y))\preceq (fs,y).$$ This means that the poset $\hat{\mathcal{P}}_A$ together with the maps $(e_B)_{B\in \mathrm{ob}\mathcal{C}}$ form a cowedge. We will now show that they form a \emph{universal} cowedge. To do this, consider another cowedge, i.e. a poset $\mathcal{R}$ toghether with order-preserving maps $(\Tilde{e}_B:\mathcal{C}(B,A)\times FB\to \mathcal{R})_{B\in \ob\mathcal{C}}$ such that for every morphism $s:B_2\to B_1$,  \[\begin{tikzcd}
	{\mathcal{C}(B_1,A)\times FB_2} && {\mathcal{C}(B_2,A)\times FB_2} \\
	\\
	{\mathcal{C}(B_1,A)\times FB_1} && {\mathcal{R}}
	\arrow["{\text{Id}\times Fs}"{description}, from=1-1, to=3-1]
	\arrow["{\mathcal{C}(s,A)\times \text{Id}}", from=1-1, to=1-3]
	\arrow["{\tilde{e}_{B_1}}"', from=3-1, to=3-3]
	\arrow["{\tilde{e}_{B_2}}", from=1-3, to=3-3]
	\arrow["\preceq"{description}, draw=none, from=3-1, to=1-3]
\end{tikzcd}\]

Let $e:\mathcal{P}_A\to \tilde{\mathcal{P}}$ be the map that sends $(B\xrightarrow{g}A,y)$ to $\tilde{e}_B(g,y)$. For $(g_1,y_1)\preceq (g_2,y_2)$ in $\mathcal{P}_A$, there exists a morphism $s:B_2\to B_1$ such that $g_1s\leq g_2$ and $y_1\leq Fs(y_2)$. Because $\tilde{e}_{B_1}$ and $\tilde{e}_{B_2}$ are order-preserving and because the maps $(\tilde{e}_B)_{B\in \mathrm{ob}\mathcal{C}}$ form a wedge, we find the following relations in $\mathcal{R}$: 

$$\tilde{e}_{B_1}(g_1,y_1)\leq \tilde{e}(g_1,Fs(y_2))\leq \tilde{e}_{B_2}(g_1s,y_2)\leq \tilde{e}_{B_2}(g_2,y_2).$$

This shows that $e$ is order-preserving, and therefore it induces an order-preserving map $\hat{e}:\hat{\mathcal{P}}_A\to \mathcal{R}$. This is the unique order-preserving map such that $e\circ e_B = \tilde{e}_B$ for all $B\in \ob\mathcal{C}$. This shows the universal property.
\end{proof}

Let $[\mathcal{C},\textbf{Pos}]_l$ be the category of enriched functors $\mathcal{C}\to \textbf{Pos}$ and lax transformations and let $[\mathcal{C},\textbf{Pos}]_s$ be the subcategory of enriched functors and strict transformations. The following result is Theorem 3.16 in \cite{blackwell} together with section 7.1.2 in \cite{loregian}. In \cite{blackwell} the result follows from a more general theorem. Here we will also give a direct proof. 
\begin{prop} \label{Blackwell}
The inclusion $[\mathcal{C},\mathbf{Pos}]_s\to [\mathcal{C},\mathbf{Pos}]_l$ has a left adjoint, which is given by the assignment $F\mapsto F^\#$.
\end{prop}
\begin{proof}
It is enough to show that $[F^\#,G]_s\cong [F,G]_l$ for all functors $F$ and $G$. 

Given a strict transformation $\tau: F^\#\to G$, we can define for every $A\in \ob\mathcal{C}$ a functor $\lambda_A:FA\to GA$ by sending $x$ to $\tau_A(1_A,x)$. These form a lax transformation. 

Given a lax transformation $\lambda:F\to G$, we can define a strict transformation $\tau: F^\#\to G$, by the assignment \[\tau_A(g:B\to A,y):=\lambda_B(Fg(y)).\] These define an isomorphism of posets. 
\end{proof}
The counit of the adjunction in Proposition \ref{Blackwell} is a strict transformation $$l':F^\#\to F$$ and is defined by $$l'_A:F^\#A\to FA: (g:B\to A,y)\mapsto Fg(y),$$ for all objects $A$ in $\mathcal{C}.$

Dually, we define the functor $F_\#:\mathcal{C}\to \textbf{Pos}$ by sending every object $A$ in $\mathcal{C}$ to the colax coend \[F^\#A:=\ointctrclockwise^{B:\mathcal{C}}\mathcal{C}(B,A)\times FB.\]

A construction dual to the one from Proposition \ref{4.2} can be given for this poset. We also have a dual version of Proposition \ref{Blackwell}, namely that the assignment $F\mapsto F_\#$ defines a left adjoint to the inclusion $[\mathcal{C},\mathbf{Pos}]_s\to [\mathcal{C},\mathbf{Pos}]_c$. The functor $F_\#$ is called the \emph{colax morphism classifier} and has been studied in \cite{blackwell,loregian,lack}.

The counit is a strict transformation $$c': F_\#\to F$$ defined by 
$$c'_A:F_\#A:FA: (g,y)\mapsto Fg(y),$$ for all objects $A$ in $\mathcal{C}$.

Finally, consider the functor $F^\#_\#:\mathcal{C}\to \textbf{Pos}$ that is defined by sending every object $A$ in $\mathcal{C}$ to \[F^\#_\#A:=\sum_{B\in \ob\mathcal{C}}\mathcal{C}(B,A)\times FB.\]

Similarly to Proposition \ref{Blackwell}, it can be shown that the assignment $$F\mapsto F_\#^\#$$ defines a left adjoint to the inclusion $[\mathcal{C},\textbf{Pos}]\to \textbf{Pos}^{\mathrm{ob}\mathcal{C}}$.

The unit $F\to F^\#$ is an element of $[F,F^\#]_l\subseteq [F,F^\#]$ and therefore it corresponds to a strict transformation $$l:F^\#_\#\to F^\#.$$ We have that $$l_A(g,y)=(g,y)$$ for all $A\in \ob\mathcal{C}$ and $(g,y)\in F^\#_\#A$.

Similarly, the unit $F\to F_\#$ is an element of $[F,F^\#]_c\subseteq [F,F^\#]$ and also corresponds to a strict transformation $$c:F^\#_\#\to F_\#$$ such that $$c_A(g:B\to A,y)=(g,y)$$ for all $A$ in $\mathcal{C}$ and $(g,y)$ in $F^\#_\#A$.

This gives us the strict transformations that we will need in the rest of this section: 
\[\begin{tikzcd}
	& F \\
	{F_\#} && {F^\#} \\
	& {F_\#^\#}
	\arrow["{c}", from=3-2, to=2-1]
	\arrow["{l}"', from=3-2, to=2-3]
	\arrow["{c'}", from=2-1, to=1-2]
	\arrow["{l'}"', from=2-3, to=1-2]
\end{tikzcd}\]
As explained above, writing the general, lax and colax transformations in terms of $F^\#$ and $F_\#^\#$ allows us to describe the operation $\overline{(-)}$ as a left extension. For a lax transformation $\lambda:F\to G$, let $\tilde{\lambda}:F^\#\to G$ be the corresponding strict transformation. For $A\in \ob\mathcal{C}$ and $x\in FA$, \emph{assuming} that the left Kan extension exist,  we can write \[\overline{\lambda}_A(x)=\mathrm{Lan}_c(\tilde{\lambda}\circ l)_A(1_A,x).\]
\emph{Suppose} now that this left extension is objectwise. We then would have that $\overline{\lambda}_A(x)$ is equal to \begin{align*}
    & \mathrm{colim}(c\downarrow (1_A,x)\to F_\#^\#A\xrightarrow{l}F^\#A\xrightarrow{\tilde{\lambda}}GA)\\
    &= \bigvee\left\{\tilde{\lambda}_A(g,y)\mid g:C\to A; y\in FC \textrm{ such that }Fg(y)\leq x\right\}\\
    &= \bigvee\left\{Gg(\lambda_C(y))\mid g:C\to A; y\in FC\textrm{ such that } Fg(y)\leq x\right\}\end{align*}

In the following two results (Proposition \ref{4.7} and Theorem \ref{monadformula}), we will give conditions for when this is indeed the case. More specifically, we will give conditions on the functors $F$ and $G$ such that the \textit{colaxification formula}, 

$$\overline{\lambda}_A(x) =\bigvee\left\{Gg(\lambda_C(y))\mid g:C\to A; y\in FC\textrm{ such that } Fg(y)\leq x\right\},$$

holds every lax transformation $\lambda:F\to G$ , object $A$ in $\mathcal{C}$ and $x\in FA$.

\begin{prop}\label{4.7}
For $\lambda\in [F,G]^\Sigma$ and $A\in \ob\mathcal{C}$, define $\tau_A:FA\to GA$ by \[\tau_A(x):=\bigvee\left\{Gg(\lambda_C(y))\mid g:C\to A; y\in FC\textrm{ such that } Fg(y)\leq x\right\},\]
for $x\in FA$. If $\tau$ is colax, then $\tau = L_c\lambda$. If moreover, $\lambda$ is lax, then $\tau=\overline{\lambda}$.
\end{prop}
\begin{proof}
Let $\tilde{\lambda}$ be the strict transformation $F^\#_\#\to G$, corresponding to $\lambda$ and let $\tilde{\tau}$ be the strict transformation $F_\#\to G$ corresponding to $\tau$. We want to show that $\tilde{\tau}$ is the left extension of $\tilde{\lambda}$ along $c$. We first show that $\tilde{\tau}\circ c \geq \tilde{\lambda}$, i.e. 

\[\begin{tikzcd}
	{F_\#^\#} && G \\
	{F_\#}
	\arrow["c"', from=1-1, to=2-1]
	\arrow["{\tilde{\tau}}"', from=2-1, to=1-3]
	\arrow[""{name=0, anchor=center, inner sep=0}, "{\tilde{\lambda}}", from=1-1, to=1-3]
	\arrow["\geq"{marking}, draw=none, from=0, to=2-1]
\end{tikzcd}\]

For $(f:C\to A,y)\in F^\#_\#A$, \begin{align*}
    (\tilde{\tau}\circ c)_A(f,y)= Gf(\tau_C(y)) &= Gf\bigvee\left \{Gg(\lambda_{C'}(y'))\mid g:C'\to C; y'\in FC'\textrm{ such that } Fg(y')\leq y\right\}\\
     & \geq \bigvee\left \{ Gf(Gg(\lambda_{C'}(y')))\mid g:C'\to C; y'\in FC'\textrm{ such that } Fg(y')\leq y\right\}\\
     & \geq Gf(\lambda_C(y))= \tilde{\lambda}_A(f,y),
\end{align*} which means that $\tilde{\tau}\circ c\geq \tilde{\lambda}$. 

Suppose now that $\nu:F_\#\to G$ is a strict transformation such that $\nu\circ c \geq \tilde{\lambda}$, i.e. \[\begin{tikzcd}
	{F_\#^\#} && G \\
	{F_\#}
	\arrow["c"', from=1-1, to=2-1]
	\arrow["\nu"', from=2-1, to=1-3]
	\arrow[""{name=0, anchor=center, inner sep=0}, "{\tilde{\lambda}}", from=1-1, to=1-3]
	\arrow["\geq"{marking}, draw=none, from=0, to=2-1]
\end{tikzcd}\]

We will now show that $\tilde{\tau}\leq \nu$. Let $(f:C\to A,y)\in F_\#A$. Consider $g:C'\to C$ and $y'\in FC'$ such that $Fg(y')\leq y$, then $(g,y')\in F^\#_\#C$ and  $c_C(g,y')\leq (1_C,y)$ in $F_\#C$ and therefore \[\nu_C(1_C,y)\geq (\nu\circ c)_C(g,y') \geq \tilde{\lambda}_C(g,y')=Gg\lambda_{C'}(y').\] Taking the supremum over all such $(g,y')$, we have \[\nu_C(1_C,y)\geq \tau_C(y)\] and by applying $Gf$ to both sides,  \[\nu_A(f,y)=Gf\nu_A(1_C,y)\geq Gf\tau_C(y)=\tilde{\tau}_A(f,g).\] Here we used that $\nu$ is strict and the definition of $\tilde{\tau}$. This shows that $\tilde{\tau}$ is the left extension of $\tilde{\lambda}\circ l$ along $c$.
\end{proof}

Let $A\in \ob\mathcal{C}$ and $x\in FA$. Consider the subposets of $F^\#A$: $$l'_A\downarrow x := \{(g:C\to A,y)\mid Fg(y)\leq x\}$$

and 

$$(l'_A\downarrow x)_{=}:= \{(g:C\to A,y)\mid Fg(y)\leq x\}.$$


\begin{thm}[Colaxification formula] \label{monadformula}
Let $\kappa$ be a regular cardinal. Suppose that 

\begin{itemize}
    \item [($G_1)$]$Gf$ preserves $\kappa$-directed joins for all morphisms $f$ in $\mathcal{C}$,
    \item [$(F_1)$]$(l'_A\downarrow x)_{=}\subseteq l'_A\downarrow x$ is  cofinal \footnote{A subset $S$ of a poset $P$ is \textbf{cofinal} if for every $p\in P$ there exists an $s\in S$ such that $p\leq s$.},
    \item [$(F_2)$]$\mathcal{C}$ has and $F$ preserves $\kappa$-wide pullbacks.
    \end{itemize}
Then $\overline{\lambda}$ is a strict transformation for every lax transformation $\lambda:F\to G$, i.e. $(F,G,\emptyset)$ satisfies the strictness condition. 

Furthermore, the colaxification formula holds, i.e. for $A\in \ob\mathcal{C}$ and $x\in FA$, \[\overline{\lambda}_A(x)=\bigvee\left\{Gg(\lambda_C(y))\mid g:C\to A; y\in FC\textrm{ such that }Fg(y)\leq x\right\}.\]
\end{thm}
\begin{proof}
For $A\in \ob\mathcal{C}$ and $x\in FA$ denote \[S_{A,x}:=\{Gg\lambda_C(y)\mid g:C\to A; y\in FC\textrm{ such that } Fg(y)\leq x\}\]
and write $\tau_A(x):=\bigvee S_{A,x}$. We will now show that $S_{A,x}$ is $\kappa$-directed.  

Let $I$ be a set such that $\lvert I\rvert < \kappa$. And consider a collection $(g_i:C_i\to A,y_i)_{i\in I}$ in $l'_A\downarrow x$.

From $(F1)$ it follows that for every $i\in I$, there exists a $(\tilde{g}_i:\tilde{C}_i\to A,\tilde{y}_i) \in (l'_A\downarrow x)_{=}$ such that 

$$(g_i,y_i)\preceq (\tilde{g}_i,\tilde{y}_i),$$
i.e. for every $i\in I$ there is a map $s_i:\tilde{C}_i\to C_i$ such that $g_is_i\leq \tilde{g}_i$ and such that $y_i\leq Fs_i(\tilde{y}_i)$.\footnote{Note that we might need the axiom of choice here.} 

The maps $(\tilde{g}_i)_{i\in I}$ form a $\kappa$-wide pullback diagram. Let $C$ be the pullback of this diagram and let $(p_i:C\to \tilde{C}_i)_{i\in I}$ be the collection of projections. \[\begin{tikzcd}
	& \vdots &  \\
	& {\tilde{C}_i} \\
	C & \vdots & A \\
	& {\tilde{C}_j} \\
	& \vdots & 
	\arrow["{\tilde{g_j}}"', from=4-2, to=3-3]
	\arrow["{\tilde{g_i}}", from=2-2, to=3-3]
	\arrow["{p_i}", from=3-1, to=2-2]
	\arrow["{p_j}"', from=3-1, to=4-2]
	\arrow["\lrcorner"{anchor=center, pos=0.125, rotate=45}, draw=none, from=3-1, to=3-3]
\end{tikzcd}\]
Because $F\tilde{g}_i(\tilde{y}_i)= x$ for all $i\in I$, we have that $(\tilde{y}_i)_{i\in I}\in \bigtimes_{FA}FC_i$. Since $F$ preserves this limit, there exists a unique $y\in FC$ such that $Fp_i(y)=\tilde{y}_i$. Let $g:C\to A$ be the morphism that is equal to $\tilde{g}_ip_i$ for all $i\in I$. We find for all $i\in I$ that \begin{align*}
    Gg_i(\lambda_{C_i}(y_i))& \leq Gg_i(\lambda_{C_i}(Fs_i(\tilde{y}_i)))\\
    & =Gg_i(\lambda_{C_i}(Fs_iFp_i(y)))\\
    &\leq (Gg_is_ip_i)(\lambda_C(y))\\
    & \leq (G\tilde{g}_ip_i)(\lambda_C(y))= Gg\lambda_C(y)
\end{align*}
We conclude that $S_{A,x}$ is $\kappa$-directed. Using $(G_1)$, for a morphism $f:A\to B$, we see that \begin{align*}
    Gf\tau_A(x) & = \bigvee\left\{Gf(Gg(\lambda_{C}(y)))\mid g:C\to A; y\in FC\textrm{ such that } Fg(y)\leq x\right\}\\
    & \leq \bigvee S_{B,Ff(x)}= \tau_B(Ff(x))).\end{align*}
This shows that $\tau$ is a colax transformation. It follows from Proposition \ref{4.7} that $\overline{\lambda}=\tau$. This shows the second claim. 

To prove the first claim we need to show that for $f:A\to B$ \[L:=\bigvee\left\{Gf(Gg(\lambda_{C}(y)))\mid g:C\to A; y\in FC\textrm{ such that } Fg(y)\leq x\right\}\geq \bigvee S_{B,Ff(x)}\]

Consider $(g:C\to B, y)$ in $l'_B\downarrow Ff(x)$. Then by $(F_1)$ there exist $(\tilde{g}:\tilde{C}\to B,\tilde{y})\in (l'_B\downarrow Ff(x))_{=}$ such that $(g,y)\preceq (\tilde{g},\tilde{y}),$
i.e. there is a map $s:\tilde{C}\to C$ such that $gs\leq \tilde{g}$ and $y\leq Fs(\tilde{y})$.


Let $P$ be the pullback of $f$ and $\tilde{g}$ and let $p_1:P\to A$ and $p_2:P\to \tilde{C}$ be the projection maps. \[\begin{tikzcd}
	P && A \\
	\\
	{\tilde{C}} && B
	\arrow["f"{description}, from=1-3, to=3-3]
	\arrow["{\tilde{g}}"{description}, from=3-1, to=3-3]
	\arrow["{p_2}"{description}, from=1-1, to=3-1]
	\arrow["{p_1}"{description}, from=1-1, to=1-3]
	\arrow["\lrcorner"{anchor=center, pos=0.125}, draw=none, from=1-1, to=3-3]
\end{tikzcd}\]
Clearly $(x,\tilde{y})$ is an element of $FA\times_{FB}F\tilde{C}$. Because $F$ preserves pullbacks, there exists a unique $z\in F(A\times_B\tilde{C})$ such that $Fp_1(z)=x$ and $Fp_2(z)=\tilde{y}$.  We now have \begin{align*}
    Gg(\lambda_C(y)) &\leq Gg(\lambda_C(Fs(\tilde{y}))) \\
    & = Gg(\lambda_C(FsFp_2(z)))\\
    & \leq (Ggsp_2)(\lambda_P(z))\\
    & \leq( G\tilde{g}p_2)(\lambda_P(z))\\
    & = Gf(Gp_1(\lambda_P(z)))\leq L.
\end{align*}
Since $(g,y)$ was arbitrary, this shows that $\bigvee S_{B,Ff(x)}\leq L$. This proves that $\overline{\lambda}$ is a strict transformation.
\end{proof}

\begin{rem}
    In the case of Theorem \ref{monadformula}, $(F,G,\emptyset)$ satisfies the strictness conditions. Therefore, Proposition \ref{4.1} can be applied. \end{rem}
\section{The extension formula}\label{5}
In this section we will give sufficient conditions for extensions to be objectwise. This will allow us to give concrete formulas for the extensions of $\Sigma$-natural (co)lax transformations that were discussed in Section \ref{3}. In Section \ref{7} we will use these constructions and result to extend premeasures to measures, using the representation of (pre)measures as strict transformations described in Section \ref{6}. 

Let $\mathcal{C}$ be a small poset-enriched category and let $\Sigma$ be a collection of morphisms in $\mathcal{C}$. Let $F,G,H$ be enriched functors $\mathcal{C}\to \textbf{Pos}$. Suppose that $GA$ is a complete for all $A\in \ob\mathcal{C}$ and that $Gf$ preserves meets and joins for all $f\in \Sigma$ and let $\iota:F\to H$ be a strict transformation.

We will start with a result on the objectwiseness of extensions of $\Sigma$-natural \emph{general} transformations, this is the content of Theorem \ref{5.1}. In theorem \ref{objectwisesigmacolax}, we will give sufficient conditions for $\Sigma$-natural colax transformations to be objectwise.

\begin{thm}\label{5.1}
Suppose that $Ff$ and $Hf$ have left adjoints $(Ff)_*$ and $(Hf)_*$ for all $f:A\to B$ in $\Sigma$, such that the following square commutes\[\begin{tikzcd}
	FA && FB \\
	\\
	HA && HB
	\arrow["{\iota_A}"{description}, shift right=2, from=1-1, to=3-1]
	\arrow["{\iota_B}"{description}, from=1-3, to=3-3]
	\arrow["{(Ff)_*}"{description}, from=1-3, to=1-1]
	\arrow["{(Hf)_*}"{description}, from=3-3, to=3-1]\end{tikzcd}\]
 Then left extensions of $\Sigma$-natural transformations along $\iota$ are objectwise. 
\end{thm}
\begin{proof}
Let $\tau:F\to G$ be a $\Sigma$-natural transformation. It is enough to show that $(\mathrm{Lan}_{\iota_A}\tau_A)_{A\in \ob\mathcal{C}}$ is a $\Sigma$-natural general transformation. Let $f:A\to B$ be in $\Sigma$ and let $x\in HA$, 
\begin{align*}
    Gf \mathrm{Lan}_{\iota_A}\tau_A(x) &= Gf\mathrm{colim}(\iota_A\downarrow x \to FA\xrightarrow{\tau_A}GA) \\
    & =\mathrm{colim}(\iota_A\downarrow x \to FA\xrightarrow{\tau_A}GA\xrightarrow{Gf}GB)\\
    & =\mathrm{colim}(\iota_A\downarrow x \to FA\xrightarrow{Ff}FB\xrightarrow{\tau_B}GB)\\
    & = \mathrm{colim}(\iota_A\downarrow x\to \iota_B\downarrow Hf(x) \to FB\xrightarrow{\tau_B}GB)\\
    & = \mathrm{colim}(\iota_B\downarrow Hf(x) \to FB\xrightarrow{\tau_B}GB)=\mathrm{Lan}_{\iota_B}\tau_B(Hf(x))
\end{align*}
The second equality follows from the fact that $Gf$ preserves joins, since $f\in \Sigma$. Because $Gf\tau_A=\tau_BFf$, we have the third equality. Equality four, follows from the fact that $\iota$ is a strict transformation.

Using the hypothesis, it can be shown that \[\iota_A\downarrow x \to \iota_B\downarrow Hf(x) \cong \iota_A(Ff)_*\downarrow x\] is right adjoint and therefore final. The fourth equality now follows.
\end{proof}
\begin{thm}\label{objectwisesigmacolax}
Let $\kappa$ be a regular cardinal.
\begin{itemize}
    \item [$(G1)$] Suppose that $Gf$ preserves $\kappa$-directed joins for all morphisms $f$ in $\mathcal{C}$.
    \item [$(\iota1)$] Suppose that $\iota_A$ is $\kappa$-flat for all $A\in \ob\mathcal{C}$, i.e. $\iota_A\downarrow x$ is $\kappa$-directed for all $x\in FA$.
    \item [$(\iota2)$] $Ff$ and $Hf$ have left adjoints $(Ff)_*$ and $(Hf)_*$ for all $f:A\to B$ in $\Sigma$, such that the following square commutes \[\begin{tikzcd}
	FA && FB \\
	\\
	HA && HB
	\arrow["{\iota_A}"{description}, shift right=2, from=1-1, to=3-1]
	\arrow["{\iota_B}"{description}, from=1-3, to=3-3]
	\arrow["{(Ff)_*}"{description}, from=1-3, to=1-1]
	\arrow["{(Hf)_*}"{description}, from=3-3, to=3-1]\end{tikzcd}\]
\end{itemize}
Then left extensions of $\Sigma$-natural colax transformations along $\iota$ are objectwise.
\end{thm}
\begin{proof}
Let $\sigma:F\to G$ be a $\Sigma$-natural colax transformation. It is enough to show that $(\mathrm{Lan}_{\iota_A}\sigma_A)_{A\in \ob\mathcal{C}}$ is a $\Sigma$-natural colax transformation. Let $f:A\to B$ be a morphism in $\ob\mathcal{C}$ and let $x\in HA$, \begin{align*}
    Gf \mathrm{Lan}_{\iota_A}\sigma_A(x) &= Gf\mathrm{colim}(\iota_A\downarrow x \to FA\xrightarrow{\sigma_A}GA) \\
    & =\mathrm{colim}(\iota_A\downarrow x \to FA\xrightarrow{\sigma_A}GA\xrightarrow{Gf}GB)\\
    & \leq \mathrm{colim}(\iota_A\downarrow x\to \iota_B\downarrow Hf(x) \to FB\xrightarrow{\sigma_B}GB)\\
    & \leq \mathrm{colim}(\iota_B\downarrow Hf(x) \to FB\xrightarrow{\sigma_B}GB)=\mathrm{Lan}_{\iota_B}\sigma_B(Hf(x))
\end{align*}
In the second equality we use that $\iota_A\downarrow x$ is $\kappa$-directed and that $Gf$ preserves $\kappa$-directed joins. The inequality in the third line follows from the fact that $\sigma$ is colax and the last inequality is induced by the inclusion $\iota_A\downarrow x \to \iota_B\downarrow Hf(x)$. This shows that it is a colax transformation. 

To show that the colax transformation is a $\Sigma$-natural, suppose that $f\in \Sigma$. In this case the first inequality becomes an equality because $\sigma$ is a $\Sigma$-natural transformation and therefore $Gf\sigma_A=\sigma_BFf$. From $(\iota2)$ it is straightforward to check that \[\iota_A\downarrow x \to \iota_B\downarrow Hf(x) \cong \iota_A(Ff)_*\downarrow x\] is right adjoint and therefore final. It follows that the second inequality is an equality in this case. 
\end{proof}
\section{Premeasures as transformations}\label{6}
In this section we will represent certain premeasures by certain transformations between functors. Roughly speaking, colax and lax transformation correspond to inner and outer premeasures and strict transformations correspond to measures. Using these representations, we can apply the results from Section \ref{2} and Section \ref{5} to obtain results and formulas for joins and meets of certain premeasures and for operations to turn a certain kind of premeasure in a different kind in a universal way. 

We will study two approaches. In the first approach we will look at functors whose domain is $\mathbf{Set}_c^f$, a subcategory of countable sets and functions. In the second approach we use the category $\mathbf{Part}_c^f$, a subcategory of countable sets and \emph{partial} functions. The advantage of the first way of representing premeasures is that $\mathbf{Set}_c^f$ is slightly easier to work with than $\mathbf{Part}_c^f$. The second approach is more flexible to obtain variations of the representation result. For example, if we restrict to \emph{probability} premeasures, we need to use the second point of view to represent inner and outer probability premeasures. 

In this section we will use the notation $\bigcupdot$ and $\cupdot$ to emphasize that we are working with \emph{disjoint} unions of subsets.

\subsection{Premeasures}
We will discuss a first representation theorem for premeasures as \emph{strict} transformations from a functor $F'_{\mathcal{B}}$, describing premeasurable partitions, to a functor $G$ describing real values (Theorem \ref{rep}). The rest of this section is focused on proving that $(F'_\mathcal{B},G,\emptyset)$ satisfies the strictness condition and that the colaxification formula holds. For this we use Theorem \ref{monadformula}.

A \textbf{premeasurable space } is a set $X$ set together with an algebra  $\mathcal{B}$ of subsets, i.e. a collection of subsets closed under finite unions and complements. Let $(X_1,\mathcal{B}_1)$ and $(X_2, \mathcal{B}_2)$ be premeasurable spaces. A map $f:X_1\to X_2$ is called \textbf{premeasurable} if $f^{-1}(\mathcal{B}_2)\subseteq \mathcal{B}_1$. We denote the category of premeasurable spaces and premeasurables maps by $\mathbf{PreMble}$ and interpret it as a poset-enriched category where the hom-categories are discrete.

For a countable set $A$, let $\mathcal{P}_f(A)$ be the algebra of finite and cofinite sets. A map $f:A\to B$ is called \textbf{finite} if $f^{-1}(\mathcal{P}_f(B))\subseteq \mathcal{P}_f(A)$. Let $\mathbf{Set}_c^f$ be the category of countable sets and finite maps, which we interpret as a poset-enriched category. Endowing a countable set $A$ with the algebra $\mathcal{P}_f(A)$ of finite and cofinite sets, yields a premeasurable space. This gives a functor $i:\textbf{Set}_c^f\to \textbf{PreMble}$. Now define the enriched functor $F'_\mathcal{B}$ as \[\textbf{Set}_c^f\xrightarrow{\textbf{PreMble}((X,\mathcal{B}),i-)} \textbf{Pos}.\] For a countable set $A$, $F'_\mathcal{B}A$ can be identified with the set of $A$-indexed premeasurable partitions of $X$. Note that it is import that we restrict to \emph{finite} maps, because an arbitrary countable union of premeasurable subsets might not be premeasurable.

Let $A$ be a countable set. The set $[0,\infty]^A$ together with the pointwise order, is a poset which we will denote by $G'A$. For a finite map $f:A\to B$ of countable sets, the assignment \[(p_a)_{a\in A}\mapsto \left(\sum_{f(a)=b}p_a\right)_{b\in B}\] defines an order-preserving map $G'f:G'A\to G'B$. We obtain a strict 2-functor $G':\textbf{Set}_c^f\to \textbf{Pos}$.

A \textbf{premeasure} on a premeasurable space $(X,\mathcal{B})$ is a map $\mu:\mathcal{B}\to [0,\infty]$ such that $\mu\left(\bigcupdot_{n=1}^\infty B_n\right)=\sum_{n=1}^\infty\mu(B_n)$ for pairwise disjoint $(B_n)_{n=1}^\infty$ in $\mathcal{B}$ such that $\bigcupdot_{n=1}^\infty B_n$ is also an element of $\mathcal{B}$. For premeasures $\mu_1$ and $\mu_2$, we write $\mu_1\leq \mu_2$ if $\mu_1(B)\leq \mu_2(B)$ for all $B\in \mathcal{B}$. This turns the set of premeasures on $(X,\mathcal{B})$ into a poset, which we will denote by $M(X,\mathcal{B})$.

We start by giving a first representation theorem for premeasures, as strict transformations (Theorem \ref{rep}). This representation result is related to the ideas in section 3 of \cite{staton}. 

For a premeasure $\mu$ and a countable set $A$, define $\tau^\mu_A:F'_\mathcal{B}A\to G'A$ by the assignment \[\left(B_a\right)_{a\in A}\mapsto \left(\mu(B_a)\right)_{a\in A}. \]
\begin{prop}\label{6.1}
Let $\mu$ be a premeasure, then $(\tau^\mu_A)_{A}$ is a strict transformation $F'_\mathcal{B}\to G'$.
\end{prop}
\begin{proof}
Let $f:A\to B$ be a finite map. For $(E_a)_{a\in A}\in F'_\mathcal{B}A$ and $b\in B$ \[G'f(\tau^\mu_A((E_a)_a))_b=\sum_{f(a)=b}\mu(E_a)=\mu\left(\bigcupdot_{f(a)=b}E_a\right).\]
Here we used the $\sigma$-additivity of $\mu$.
\end{proof}
If $\mu_1\leq \mu_2$, then clearly $\tau^{\mu_1}\leq \tau^{\mu_2}$. It follows that the assignment $\mu\mapsto \tau^\mu$ defines an order-preserving map $\tau^{(-)}:M(X,\mathcal{B})\to [F'_\mathcal{B},G']_s$.

\begin{thm}\label{rep}
The map $\tau^{(-)}$ is an isomorphism of posets.
\end{thm}
\begin{proof}
We will construct an inverse map for $\tau^{(-)}$. Let $\tau:F'_\mathcal{B}\to G'$ be a strict transformation. For $n\geq 1$, write $\mathbf{n}:=\{0,1,\ldots, n-1\}$. For $B\in \mathcal{B}$, define a map $\mu_\tau:\mathcal{B}\to [0,\infty]$ by \[\mu_\tau(B):=\tau_\textbf{2}((B^c,B))_1.\]
We will now show that this map is a premeasure. Let $t:\textbf{1}\to \textbf{2}$ be the map that sends $0$ to $0$. We find \[\mu_\tau(\emptyset)=\tau_\textbf{2}((X,\emptyset))_1= \tau_\textbf{2}(F'_\mathcal{B}t(X))_1=G't(\tau_\textbf{1}(X))_1=0.\]

Now consider $(B_n)_{n=1}^\infty$ pairwise disjoint elements in $\mathcal{B}$ such that $B:=\bigcupdot_{n=1}^\infty B_n$ is also an element in $\mathcal{B}$. Writing $B_0:=B^c$, yields an element $(B_n)_{n\in \mathbb
N}\in F'_\mathcal{B}\mathbb{N}$. For a natural number $n\geq 1$, let $s_n:\mathbb{N}\to \textbf{2}$ be the finite map that sends $n$ to $1$ and every other element to $0$. Furthermore, define a map $s_0:\mathbb{N}\to \textbf{2}$ that sends $0$ to $0$ and every other element to $1$. This map is also finite. 

For $n\in\mathbb{N}$, we have a commutative diagram \[\begin{tikzcd}
	{F'_\mathcal{B}\mathbb{N}} && {G'\mathbb{N}} \\
	\\
	{F'_\mathcal{B}\textbf{2}} && {G'\textbf{2}}
	\arrow["{F's_n}"{description}, from=1-1, to=3-1]
	\arrow["{'Gs_n}"{description}, from=1-3, to=3-3]
	\arrow["{\tau_\mathbb{N}}"{description}, from=1-1, to=1-3]
	\arrow["{\tau_\textbf{2}}"{description}, from=3-1, to=3-3]
\end{tikzcd}\]
It follows that \begin{align*}
    \mu_\tau(B) & = \tau_\textbf{2}(F'_\mathcal{B}s_0((B_m)_{m\in \mathbb{N}}))_1\\
    & = G's_0(\tau_\mathbb{N}((B_m)_{m\in \mathbb{N}}))_1\\
    & = \sum_{n=1}^\infty \tau_\mathbb{N}((B_m)_{m\in \mathbb{N}})_n. 
\end{align*}
For $n\geq 1$, we find \begin{align*}
    \mu_\tau(B_n) & = \tau_\textbf{2}(F'_\mathcal{B}s_n((B_m)_{m\in \mathbb{N}}))_1\\
    & = G's_n(\tau_\mathbb{N}((B_m)_{m\in\mathbb{N}})))_1\\
    & = \tau_\mathbb{N}((B_m)_{m\in \mathbb{N}}))_n.
\end{align*}
Combining the above gives $\sum_{n=1}^\infty \mu_\tau(B_n)=\mu_\tau(B)$, which shows that $\mu$ is a premeasure. Sending $\tau$ to $\mu_\tau$ defines an order-preserving map $\mu_{(-)}:[F'_\mathcal{B},G']_s\to M(X,\mathcal{B})$. We will now show that this is the inverse of $\tau^{(-)}$.

Let $A$ be a countable set and let $a_0\in A$. Consider the finite map $s_{a_0}:A\to \textbf{2}$ that sends $a_0$ to $1$ and every other element to $0$. For a strict transformation $\tau:F'_\mathcal{B}\to G'$ and for $(B_a)_{a\in A}\in F'_\mathcal{B}A$, \[\tau^{\mu_\tau}_A((B_a)_a))_{a_0}=\mu_\tau(B_{a_0})=\tau_\textbf{2}(F'_\mathcal{B}s_{a_0}((B_a)_a)_1=G's_{a_0}(\tau_A((B_a)_a))_1=\tau_A((B_a)_a)_{a_0}\]
This shows that $\tau^{\mu_\tau}=\tau$. For a premeasure $\mu$ and $B\in \mathcal{B}$, \[\mu_{\tau^\mu}(B)=\tau^\mu_\textbf{2}(B^C,B)_1=\mu(B).\] We conclude that $\mu_{\tau^\mu}=\mu$.
\end{proof}

We want to apply Theorem \ref{monadformula} to the transformations that represent premeasures. The following two results (Lemma \ref{lemma} and Proposition \ref{proppreserves}) will help us to verify the conditions of Theorem \ref{monadformula}. This will lead to the proof of Proposition \ref{6.5} that describes how to turn a lax transformation $F_\mathcal{B}'\to G'$ into a premeasure.  

\begin{lem} \label{lemma}
The category $\mathbf{Set}_c^f$ has pullbacks and $i:\textbf{Set}_c^f\to \textbf{PreMble}$ preserves those.
\end{lem}
\begin{proof}
Let $s:A\to C$ and $t: B\to C$ be finite maps of countable sets. Define $P:=\{(a,b)\mid s(a)=t(b)\}$ and let $p_1:P\to A$ and $p_2:P\to B$ be the projection maps. Because $p_1^{-1}(\{a\})= t^{-1}(s(a))$ for all $a\in A$, $p_1$ is a finite map. Similarly, $p_2$ is a finite map. It is now easy to see that $(P,p_1,p_2)$ forms the pullback of $s$ and $t$ in $\mathbf{Set}_c^f$.

To show that $i$ preserves this pullback, we need to show that the algebra generated by $p_1$ and $p_2$ contains all singletons. But this follows immediately from the fact that $p_1^{-1}(\{a\})\cap p_2^{-1}(\{b\})=\{(a,b)\}$ for all $(a,b)\in P$.
\end{proof}

\begin{prop}\label{proppreserves}
For $f:A\to B$ in $\mathbf{Set}_c^f$, $G'f$ preserves directed joins.    
\end{prop}
\begin{proof}
Let $(p^i)_{i\in I}$ be a directed collection of elements in $GA$. For $a\in A$, define $p_a:=\sup_{i\in I}p_a^i$. For $b\in B$,  we always have that \[G'f((p_a)_a)_b=\sum_{f(a)=b}p_a\geq \sup_{i\in I}\sum_{f(a)=b}p_a^i=\sup_{i\in I}G'f(p^i).\]
To prove the other inequality, consider an element $b\in B$. Suppose first that there exists $a\in f^{-1}(\{b\})$ such that $p_a=\infty$. In that case, for every $N\geq 1$, there exists $i_N\in I$ such that $N\leq p_a^{i_N}$. In that case \[N\leq p_a^{i_N}\leq \sum_{f(a)=b}p_a^{i_N}\leq \sup_{i\in I}\sum_{f(a)=b}p_a^i.\] By taking $N\to \infty$, we conclude that $\sup_{i\in I}\sum_{f(a)=b}p_a^i=\infty$, which shows the equality.

Suppose now that $p_a<\infty$ for all $a\in f^{-1}(\{b\})$ and consider a finite set $A'\subseteq f^{-1}(\{b\})$ and $\epsilon>0$. For every $a\in A'$ there exists $i_a\in I$ such that \[p_a-\frac{\epsilon}{\lvert A'\rvert}\leq p_a^{i_a}.\]
Because $(p^i)_{i\in I}$ is directed and $A'$ is finite, there exists an $i_0\in I$ such that $p^{i_a}\leq p^{i_0}$ for all $a\in A'$. Summing over $A'$ leads to \[\sum_{a\in A'}p_a-\epsilon\leq \sum_{a\in A'}p_a^{i_a}\leq \sum_{a\in A'}p_a^{i_0}\leq \sum_{f(a)=b}p_a^{i_0} \leq \sup_{i\in I}\sum_{f(a)=b}p_a^i.\]
The claim now follows by first taking $\epsilon\to 0$ and then taking the supremum over all finite subsets of $f^{-1}(\{b\})$. 
\end{proof}

\begin{prop} \label{6.5}
The triple $(F'_\mathcal{B},G',\emptyset)$ satisfies the strictness property. Moreover, for $\lambda\in [F'_\mathcal
B,G']_l$, \[\mu^{\overline{\lambda}}(B)=\sup\left\{\sum_{a\in A}\lambda_A(B_a)\mid \bigcupdot_{a\in A}B_a=B\right\},\]
for all $B\in \mathcal{B}$.
\end{prop}
\begin{proof}
By Proposition \ref{proppreserves}, $G'f$ preserves directed joins for all $f\in \textbf{Set}_c^f$. For a countable set $A$ and $x\in F'_{\mathcal{B}}(A)$,  $(l'_A\downarrow x)_==l'_A\downarrow x$, since $F'_{\mathcal{B}}(A)$ is a discrete poset. Finally,  Lemma \ref{lemma} and the fact that hom-functors preserve pullbacks imply that $F'_\mathcal{B}$ preserves pullbacks. The result now follows by Theorem \ref{monadformula} for $\kappa=\aleph_0$.
\end{proof}
\subsection{Outer/inner premeasures}\label{section6.2}
In this subsection we will give a second representation theorem for premeasures as \emph{strict} transformations and extend this to a representation result for inner and outer measures as \emph{lax} and \emph{colax} transformations (Theorem \ref{6.7}). We will then prove that also in this case the strictness condition is satisfied and the the colaxifiction formula holds. We will again rely on Theorem \ref{monadformula} to prove this. 

Let $(X_1,\mathcal{B}_1)$ and $(X_2,\mathcal{B}_2)$ be premeasurable spaces. A partial map $f:X_1\to X_2$ is \textbf{premeasurable} if $f^{-1}(B)\in \mathcal{B}_1$ for all $B\in \mathcal{B}_2$. For premeasurable partial maps $f_1,f_2:A\to B$, we write $f_1\leq f_2$ if $\mathrm{dom}f_1\subseteq \mathrm{dom}f_2$ and $f_1(x)=f_2(x)$ for all $x\in \mathrm{dom}f_1$. The poset-enriched category of premeasurable spaces and premeasurable partial maps, which are ordered, is denoted as $\mathbf{PartPreMble}$. 

A partial map $f:A\to B$ of countable sets is called \textbf{finite} if $f^{-1}(B')\in \mathcal{P}_f(A)$ for all $B'\in \mathcal{P}_f(B)$. We write $\mathbf{Part}_c^f$ for the poset-enriched category of countable sets and partial finite maps, which are ordered. The collection of injective partial finite maps is denoted by $\Sigma$. Sending a countable set $A$ to $(A,\mathcal{P}_f(A))$ again yields an enriched functor $j:\textbf{Part}_c^f\to \textbf{PartPreMble}$. For a premeasurable space $(X,\mathcal{B})$, define the enriched functor $F_\mathcal{B}$ as \[\textbf{Part}_c^f\xrightarrow{\textbf{PartPreMble}(X,\mathcal{B}),j-)}\textbf{Pos}.\] For a countable set $A$, the poset $F_\mathcal{B}A$ is the set of $A$-indexed collections of pairwise disjoint premeasurble subsets, whose union is again premeasurable, endowed with the pointwise order.

Let $A$ be a countable set. Let $GA$ be the poset $[0,\infty]^A$, where the order is defined pointwise. For a finite partial map $f:A\to B$, sending $(p_a)_{a\in A}$ to \[\left(\sum_{f(a)=b}p_a\right)_{b\in B}\] gives an order-preserving map $Gf:GA\to GB$. This yields an enriched functor $G:\textbf{Part}_c^f\to \textbf{Pos}$. 

An order-preserving map $\mu:\mathcal{B}\to [0,\infty]$ such that $\mu(\emptyset)=0$ is called an \textbf{outer premeasure} if $\mu(B)\leq \sum_{n=1}^\infty\mu(B_n)$ for pairwise disjoint subsets $(B_n)_{n=1}^\infty$ in $\mathcal{B}$ such that $B:=\bigcupdot_{n=1}^\infty B_n$ is also an element in $\mathcal{B}$. It is called an \textbf{inner premeasure} if $\mu(B)\geq \sum_{n=1}^\infty\mu(B_n)$ for pairwise disjoint subsets $(B_n)_{n=1}^\infty$ in $\mathcal{B}$ such that $B:=\bigcupdot_{n=1}^\infty B_n$ is also an element in $\mathcal{B}.$ The set of outer (inner) premeasures together with objectwise order is denoted by $M_\leq(X,\mathcal{B})$ ($M_\geq(X,\mathcal{B})$).

We will now give a second representation result that describes inner (outer) premeasures as lax (colax) transformations (Theorem \ref{6.7}).

For an order-preserving map $\mu:\mathcal{B}\to [0,\infty]$ such that $\mu(\emptyset)=0$, we define maps $(\tau^\mu_A:F_{\mathcal{B}}A\to GA)_A$ by the assignment \[(E_a)_{a\in A}\mapsto \mu(E_a)_{a\in A}.\]

\begin{prop}
If $\mu$ is an outer (inner) premeasure, then $\tau^\mu$ is a $\Sigma$-natural lax (colax) transformation.
\end{prop}
\begin{proof}
For a partial finite map $f:A\to B$ and $(E_a)_{a\in A}$ and $b\in B$, we have \[Gf(\tau^\mu_A((E_a)_a))_b = \sum_{f(a)=b}\mu(E_a) \geq \mu\left(\bigcupdot_{f(a)=b}E_a\right) = \tau^\mu(F_\mathcal{B}f((E_a)_a)_b.\]
If $f$ is injective, then $f^{-1}(\{b\})$ is either a singleton or the empty set; in both cases the inequality becomes an equality.
\end{proof}

Clearly, if $\mu$ is a premeasure, $\tau^\mu$ is a strict transformation. This leads to following order-preserving maps. \[\begin{tikzcd}
	{[F_\mathcal{B},G]^\Sigma_l} & {[F_\mathcal{B},G]_s} & {[F_\mathcal{B},G]^\Sigma_c} \\
	\\
	{M_{\leq}(X,\mathcal{B})} & {M(X,\mathcal{B})} & {M_{\geq}(X,\mathcal{B})}
	\arrow[from=1-2, to=1-1]
	\arrow[from=1-2, to=1-3]
	\arrow[from=3-2, to=3-1]
	\arrow[from=3-2, to=3-3]
	\arrow["{\lambda^{(-)}}", from=3-1, to=1-1]
	\arrow["{\tau^{(-)}}"{description}, from=3-2, to=1-2]
	\arrow["{\sigma^{(-)}}"', from=3-3, to=1-3]
\end{tikzcd}\]
\begin{thm}\label{6.7}
The maps $\lambda^{(-)}, \tau^{(-)}$ and $\sigma^{(-)}$ are isomorphisms of posets.
\end{thm}
\begin{proof}
We will only give the proof for $\lambda^{(-)}$, by constructing an inverse map.

Let $\lambda:F_\mathcal{B}\to G$ be a $\Sigma$-natural lax transformation. For $n\geq 1$, define $\mathbf{n}:=\{0,1,\ldots,n-1\}$. For $B\in \mathcal{B}$, define \[\mu_\lambda(B):=\lambda_\textbf{1}(B).\]
We will now show that this is an outer premeasure. Let $s$ be the unique injective map $\emptyset\to \textbf{1}$ and let $*$ be the unique element in $F_\mathcal{B}\emptyset$. We find \[\mu_\lambda(\emptyset)= \lambda_\textbf{1}(\emptyset)=\lambda_\textbf{1}(F_\mathcal{B}(*))= Gs(\lambda_\emptyset(*))=0.\]

Let $x:=(B_n)_{n\in \mathbb{N}}$ be a collection of pairwise disjoint premeasurable subsets such that $B:=\bigcupdot_{n\in \mathbb{N}}B_n$ is also an element of $\mathcal{B}$. Let $s:\mathbb{N}\to \textbf{1}$ be the map that sends every element to $0$ and for $m\in \mathbb{N}$, let $s_m:\mathbb{N}\to \textbf{1}$ be the partial map that is only defined on $\{m\}$. Note that these maps are finite partial maps and that $s_m$ is injective for all $m\in \mathbb{N}$. We have \begin{align*}
    \mu_\lambda(B) & = \lambda_\textbf{1}(F_\mathcal{B}s(x))\\
    & \leq Gs(\lambda_\mathbb{N}(x)) = \sum_{n\in \mathbb{N}}\lambda_\mathbb{N}(x)_n \end{align*} and for every $n\in \mathbb{N}$, \[\lambda_\mathbb{N}(x)_n=Gs_n(\lambda_\mathbb{N}(x))\\
    = \lambda_\textbf{1}(F_\mathcal{B}s_n(x))=\lambda_\textbf{1}(B_n)=\mu_\lambda(B_n).\]
This shows that $\mu_\lambda$ is an outer premeasure and this gives an order-preserving map $\mu_{(-)}:[F_\mathcal{B},G]_l^\Sigma\to M_\leq (X,\mathcal{B})$. We will now show that this map is the inverse of $\lambda^{(-)}$.

Let $\lambda\in [F_\mathcal{B},G]_l^\Sigma$ and let $A$ be a countable set. Let $a_0\in A$ and let $x:=(B_a)_{a\in A }\in F_\mathcal{B}A$. Let $s_{a_0}:A\to \textbf{1}$ be the finite partial map that is only defined on $\{a_0\}$. This map is injective. From the equalities, \[\lambda^{\mu_\lambda}(x)_{a_0}=\mu_\lambda(B_{a_0})=\lambda_\textbf{1}(F_\mathcal{B}s_{a_0}(x))=\lambda_A(x)_{a_0},\]
we conclude that $\lambda^{\mu_\lambda}=\lambda$. 

For an outer premeasure $\mu$ , we have $\mu_{\lambda^{\mu}}(B)=\lambda_\textbf{1}^\mu(B)=\mu(B)$ for all $B\in \mathcal{B}$. 

\begin{rem}\label{remark}
 From the proof of Theorem \ref{6.7} we can see that we could also restrict $\Sigma$ to collection of all partial maps of the form $$s_{a_0}:A\to \textbf{1},$$ where $A$ is a countable set, $a_0$ is an element of $A$ and the domain of $s_{a_0}$ is $\{a_0\}$. We will use this restriction in Proposition \ref{7.5}. 
\end{rem}
We conclude that $\mu_{(-)}$ is the inverse of $\lambda^{(-)}$. The claim for strict transformations and $\Sigma$-natural colax transformations follows from a similar argument.
\end{proof}

We will now use the second representation to give a formula for joins of premeasures (Proposition \ref{6.11}) and a formula that turns an inner premeasure into a premeasure (Proposition \ref{Rl} and Corollary \ref{6.13}). To do this, we again need Theorem \ref{monadformula} and the dual of Proposition \ref{4.7}. Therefore we first need to verify the conditions of Theorem \ref{monadformula}. We do this in Lemma \ref{lemma2} and Proposition \ref{proppreserves2}.

\begin{lem}\label{lemma2}
The category $\mathbf{Part}_c^f$ has pullbacks and $j:\textbf{Part}_c^f\to \textbf{PartPreMble}$ preserves these.
\end{lem}
\begin{proof}
Let $\mathbf{1}$ be the set with one element. It can be shown that $\mathbf{1}/\textbf{Set}_c^f\cong \textbf{Part}_c^f$ and that $\mathbf{1}/\textbf{PreMble}\cong \textbf{PartPreMble}$. Under these equivalences, $j\cong \textbf{1}/i$, where $i:\textbf{Set}_c^f\to \textbf{PreMble}$. The claim now follows from Lemma \ref{lemma}.
\end{proof}
\begin{prop}\label{proppreserves2}
For $f:A\to B$ in $\mathbf{Part}_c^f$, $Gf$ preserves directed joins.
\end{prop}
\begin{proof}
The proof is the same as the proof of Proposition \ref{proppreserves}.
\end{proof}
Using the previous results (Lemma \ref{lemma2} and Proposition \ref{proppreserves2}) we can now obtain a formula to turn an outer premeasure into a premeasure in a universal way, by applying Theorem \ref{monadformula}.
\begin{prop}\label{6.10}
The triple $(F_\mathcal{B},G,\Sigma)$ satisfies the strictness condition. Moreover, the map ${\color{OliveGreen}M_\leq(X,\mathcal{B})\cong [F_\mathcal{B},G]_l^\Sigma\to[F_\mathcal{B},G]_s\cong M(X,\mathcal{B})}$, sends an outer premeasure $\mu$ to $\overline{\mu}$, which is defined as \[\overline{\mu}(B):=\sup\left\{\sum_{n=1}^\infty\mu(B_n)\mid \bigcupdot_{n=1}^\infty B_n\subseteq B, \bigcupdot_{n=1}^\infty B_n \in \mathcal{B}\right\},\] for all $B\in \mathcal{B}$.
\end{prop}
\begin{proof}
By Proposition \ref{proppreserves2}, we know that $Gf$ preserves directed joins for $f:A\to B$ in $\textbf{Part}^f_c$. From Lemma \ref{lemma2} and the fact that hom-functors preserve pullbacks, we know that $F_\mathcal{B}$ preserves pullbacks. 

Let $A$ be a countable set and let $(X_a)_{a\in A} \in F_\mathcal{B}A$. Consider $g:C\to A$ and $(Y_c)_{c\in C}\in F_\mathcal{B}C$ such that $\bigcupdot_{g(c)=a}Y_c\subseteq X_a$ for all $a\in A$. In other words $$(g:C\to A, (Y_c)_{c\in C}$$ is an element of $l'_A\downarrow (X_a)_{a\in A}$.

Because $\bigcupdot_{c\in C}Y_c\in \mathcal{B}$ and $g$ is finite, we know that $Z_a:=\bigcupdot_{g(c)=a}Y_c\in \mathcal{B}$ for all $a\in A$. For $d\in C\amalg A=:D$, define \[\tilde{Y}_d:=\begin{cases}Y_c  \ \mathrm{ if } \ d=c\\
X_a\setminus Z_a \ \mathrm{ if } \ d=a\end{cases}.\]
It is clear that $(\tilde{Y})_{d\in D}$ is an element of $F_\mathcal{B}D$. Consider the finite partial map $g+1:D\to A$, which is the same as $g$ on $C$ and the identity on $A$. Let $s:D\to C$ be the finite partial map that is the identity on $C$ and is not defined on $A$.

We have the following inequality \[\begin{tikzcd}
	C && A \\
	D
	\arrow["g", from=1-1, to=1-3]
	\arrow[""{name=0, anchor=center, inner sep=0}, "{g+1}"', from=2-1, to=1-3]
	\arrow["s", from=2-1, to=1-1]
	\arrow["\leq"{description}, draw=none, from=1-1, to=0]
\end{tikzcd}\]
and we have that $\bigcupdot_{s(d)=c}\tilde{Y}_d=Y_c$ for all $c\in C$ and \[\bigcupdot_{(g+1)(d)=a}\tilde{Y}_d=Z_a\cupdot (X_a\setminus Z_a)=X_a.\]

This means that $$(g+1:D\to A,(\tilde{Y}_d)_{d\in D})$$ is an element of $(l'_A\downarrow (X_a)_{a\in A})_=$ and that $$(g,(Y_c)_{c\in C})\preceq (g+1,(\tilde{Y}_d)_{d\in D}).$$
Therefore it follows that $(l'_A\downarrow (X_a)_{a\in A})_{=}\subseteq(l'_A\downarrow (X_a)_{a\in A})$ is cofinal. Theorem \ref{monadformula} now says that $(F_\mathcal{B},G,\Sigma)$ satisfies the strictness condition and that for $\lambda\in [F_\mathcal{B},G]_l$ and $x\in F_\mathcal{B}A$ \[\overline{\lambda}_A(x)=\bigvee\left\{Gg\lambda_C(y)\mid g:C\to A; y\in F_\mathcal{B}C; F_\mathcal{B}g\leq x\right\}.\]
Let $\mu$ be an outer premeasure, then by definition $\overline{\mu}=\mu_{\overline{\lambda^\mu}}$. 
For $B\in \mathcal{B}$ this means,
\begin{align*}
    \overline{\mu}(B)& =\overline{\lambda^\mu}_\textbf{1}(B)\\
    & = \sup\left\{Gg\lambda^\mu((B_c)_c)\mid g:C\to \textbf{1}; \bigcupdot_{c\in \mathrm{dom}g}B_c\subseteq B\right\}\\
    & = \sup \left\{\sum_{n=1}^\infty \mu(B_n)\mid \bigcupdot_{n=1}^\infty B_n\subseteq B, \bigcupdot_{n=1}^\infty B_n\in\mathcal{B}\right\}.
\end{align*}
\end{proof}
 The following proposition now immediately follows from Corollary \ref{2.4}, giving us a formula for the joins of premeasures.
\begin{prop}\label{6.11}
The poset $M(X,\mathcal{B})$ has all joins and for $(\mu_i)_i$ in $M(X,\mathcal{B})$, \[\left(\bigvee_{i\in I}\mu_i\right)(B)=\sup\left\{\sum_{n=1}^\infty \sup_{i\in I}\mu_i(B_n)\mid \bigcupdot_{n=1}^\infty B_n\subseteq  B, \bigcupdot_{n=1}^\infty B_n\in\mathcal{B}\right\}\] for $B\in \mathcal{B}$.
\end{prop}
The next result gives us a way to turn a $\Sigma$-natural general transformation into an outer measure. 
\begin{prop}\label{Rl}
Let $\sigma\in [F_\mathcal{B},G]^\Sigma$. For a countable set $A$ and $x:=(X_a)_{a\in A}\in F_\mathcal{B}A$ and $a\in A$, \[(R_l\sigma)_A(x)_a=\inf \left\{\sum_{n=1}^\infty \sigma_\textbf{1}(B_n)\mid X_a\subseteq \bigcupdot_{n=1}^\infty B_n, \bigcupdot_{n=1}^\infty B_n\in\mathcal{B}\right\}.\]
\end{prop}
\begin{proof}
For a countable set $C$ and an element $c\in C$, let $s_c:C\to \textbf{1}$ be the injective finite partial map that is only defined on $\{c\}$. For $y:=(B_c)_{c\in C}\in F_\mathcal{B}C$, we have \[\sigma_C(y)_c=Gs_c\sigma_C(y)=\sigma_\textbf{1}(B_c).\]

Define for a countable set $A$ and $x:=(X_a)_{a\in A}\in F_\mathcal{B}A$,  \begin{align*}\tau_A(x)_a& :=\left(\bigwedge\left\{Gg\sigma_C(y)\mid g:C\to A; y\in F_\mathcal{B}C; Fg(y)\geq x\right\}\right)_a\\
& = \inf\left\{\sum_{g(c')=a}\sigma_C((B_c)_c)_{c'}\mid g:C\to A; X_a\subseteq \bigcupdot_{g(c')=a}B_{c'} \right\}\\
& = \inf\left\{\sum_{n=1}^\infty \sigma_\textbf{1}(B_n)\mid X_a \subseteq \bigcupdot_{n=1}^\infty B_n, \bigcupdot_{n=1}^\infty B_n\in\mathcal{B}\right\}
\end{align*}
Consider a finite partial map $f:A\to B$ and an element $b\in B$. Write $f^{-1}(\{b\})=:\{a_1,a_2,\ldots \}$. For all $n\in \mathbb{N}$, there exists $(B_m^n)_{m=1}^\infty$ such that  $X_{a_n}\subseteq \bigcupdot_{m=1}^\infty B_m^n$ and such that \[\tau_A(x)_{a_n}+\frac{\epsilon}{2^n}\geq \sum_{m=1}^\infty \sigma_{\textbf{1}}(B^n_m)\geq \sum_{m=1}^\infty \sigma_{\textbf{1}}(B^n_m\cap X_{a_n})\]
Summing over $n$ gives \[Gf\tau_A(x)_b+\epsilon \geq \sum_{n,m}\sigma_\textbf{1}(B_m^n\cap X_{a_n})\geq \tau_B(F_\mathcal{B}f(x))_b\] for all $b\in B$. Here we used that $(B_m^n\cap X_{a_n})_{n,m}$ are pairwise disjoint premeasurable subsets, whose union equals $Ff(x)_b$. Taking $\epsilon\to 0$, shows that $\tau$ is lax. By the dual of Proposition \ref{4.7} it follows that $R_l\sigma=\tau$. 

\end{proof}
This leads to a way of turning inner premeasures into premeasures in a universal way.
\begin{cor}\label{6.13}
The map ${\color{BrickRed} M_{\geq}(X,\mathcal{B})\cong [F_\mathcal{B},G]_c^\Sigma\to [F_\mathcal{B},G]_s\cong M(X,\mathcal{B}})$, sends an inner premeasure $\mu$ to $\underline{\mu}$, which is defined as \[\underline{\mu}(B):=\inf \left\{\sum_{n=1}^\infty \mu(B_n)\mid B\subseteq \bigcupdot_{n=1}^\infty B_n, \bigcupdot_{n=1}^\infty B_n\in\mathcal{B}\right\},\] for all $B\in \mathcal{B}$.
\end{cor}
\begin{proof}
If we apply Proposition \ref{Rl} on an inner premeasure $\mu$, we obtain for $B\in \mathcal{B}$, \[ \underline{\mu}(B) :=\mu_{\underline{\sigma^\mu}}(B) = \underline{\sigma^\mu}_\textbf{1}(B)=(R_l\sigma^\mu)_\textbf{1}(B)=\inf\left\{\sum_{n=1}^\infty \mu(B_n)\mid B\subseteq \bigcupdot_{n=1}^\infty B_n, \bigcupdot_{n=1}^\infty B_n\in\mathcal{B}\right\}.\]
\end{proof}

 \section{The Carath\'eodory extension theorem}\label{7}
 In this last section, we will apply the results of extensions of transformations from sections \ref{3} and \ref{5} to the transformations that represent measures and premeasures, as described in Section \ref{6}. This will lead to a proof of the Carath\'eodory extension theorem.

 Let $(X,\mathcal{B})$ be a premeasurable space. Let $\sigma(\mathcal{B})$ be the $\sigma$-algebra generated by $\mathcal{B}$. Let $\Sigma, F_{\mathcal{B}},F_{\sigma(\mathcal{B})}$ and $G$ as in subsection \ref{section6.2}. For a countable set $A$, we have an inclusion $\iota_A:F_\mathcal{B}A\to F_{\sigma(\mathcal{B})}A$. These define a strict transformation $\iota: F_\mathcal{B}\to F_{\sigma(\mathcal{B})}$. Our goal is to prove that we can extend strict transformations $F_{\mathcal{B}}\to G$ along $\iota$ and that these extensions are proper.

We will show that right extensions of strict transformations always exist, which will imply the Carath\'eodory extension theorem. We will refer to this extension as the \emph{right Carath\'eodory extension}. Moreover, this characterizes the right Carath\'eodory extension by a universal property. We will also briefly discuss left extensions. However, we will show that the \emph{left Carath\'eodory extension} does not exist in general. 

 \subsection{The right Carath\'eodory extension}
We first look at a useful result for right extensions of $\Sigma$-natural general transformations. The following result shows a partial objectwiseness and will later be used to reduce extensions of transformations to the usual Kan extension of order-preserving maps between posets. 

We will use the notations $\mathrm{Ran}^\text{gen}$ and $\mathrm{Ran}^\text{lax}$ to emphasize that we are looking at the right Kan extensions of $\Sigma$-natural general transformations or $\Sigma$-natural lax transformation respectively. 

\begin{lem}\label{generaltransformations}
For $\tau \in [F_\mathcal{B},G]^\Sigma$, the right extension of $\tau$ along $\iota$ exists and \[(\mathrm{Ran}^{\mathrm{gen}}_\iota\lambda)_{\textbf{1}}=\mathrm{Ran}_{\iota_\textbf{1}}\lambda_{\textbf{1}}.\]
In other words, the following commutative diagram exhibits $(\mathrm{Ran}^\mathrm{gen}_\iota \lambda)_{\textbf{1}}$ as right Kan extension of $\lambda_{\textbf{1}}$ along $\iota_{\textbf{1}}$: 
\[\begin{tikzcd}
	{\mathcal{B}} && {[0,\infty]} \\
	{\sigma(\mathcal{B})}
	\arrow["{\lambda_{\textbf{1}}}", from=1-1, to=1-3]
	\arrow["{\iota_{\textbf{1}}}"', from=1-1, to=2-1]
	\arrow["{(\text{Ran}^{\text{gen}}_\iota\lambda)_{\textbf{1}}}"', from=2-1, to=1-3]
\end{tikzcd}\]
\end{lem}
\begin{proof}
We will first show that $[F_\mathcal{B},G]^\Sigma\cong \textbf{Pos}(\mathcal{B},[0,\infty])$; these are the order-preserving maps $\mathcal{B}\to [0,\infty]$. There is an order-preserving map $[F_\mathcal{B},G]^\Sigma\to \textbf{Pos}(\mathcal{B},[0,\infty])$ that sends $\tau \to \tau_\textbf{1}$. 

For an order-preserving map $f:\mathcal{B}\to [0,\infty]$, define \[\tau^f_A\left((B_a)_{a\in A}\right):=(f(B_a))_{a\in A}.\]

Let $\tau \in [F_{\mathcal{B}}, G]^\Sigma$. For a countable set $A$ and $a_0\in A$, let $s_{a_0}:A\to \textbf{1}$ be the finite partial map that is only defined on $\{a_0\}$. Since this map is injective we find that \[\tau_A((B_a)_{a\in A})_{a_0}=Gs_{a_0}(\tau_A((B_a)_{a\in A}))_1 = \tau_\textbf{1}(B_{a_0}).\]
This shows that $\tau^{\tau_\textbf{1}}=\tau$. For an order-preserving map $f:\mathcal{B}\to [0,\infty]$, we find that $\tau^f_\textbf{1}=f$. This shows that $\tau\mapsto \tau_\textbf{1}$ is an isomorphism. 

Clearly, $\textbf{Pos}(\sigma(\mathcal{B}),[0,\infty])\xrightarrow{-\circ \iota_\textbf{1}}\textbf{Pos}(\mathcal{B},[0,\infty])$ has a right adjoint. It is now easy to see that under the identification, this gives a right extension of $\lambda$ along $\iota$ and the claimed equality. 
\end{proof}

We will now describe what right extensions of $\Sigma$-natural lax transformations look like.

\begin{prop} \label{7.2}
Let $\lambda \in [F_{\mathcal{B}},G]^\Sigma_l$. The right extension of $\lambda$ along $\iota$ exists and is proper. Moreover, \[[\mathrm{Ran}^\mathrm{lax}_\iota\lambda]_A (x)_a=\inf \left\{\sum_{n=1}^\infty \lambda_\textbf{1}(B_n)\mid X_a\subseteq \bigcupdot_{n=1}^\infty B_n, \bigcupdot_{n=1}^\infty B_n\in\mathcal{B}\right\}\]
for all countable sets $A$ and $x:=(X_a)_{a\in A}\in F_{\sigma(\mathcal{B})}A$.\end{prop}
\begin{proof}
The existence follows from Proposition \ref{3.4}. The left adjoints in the following diagram commute
\[\begin{tikzcd}
	{[F_{\sigma(\mathcal{B})},G]^\Sigma_l} && {[F_{\mathcal{B}},G]^\Sigma_l} \\
	\\
	{[F_{\sigma(\mathcal{B})},G]^\Sigma} && {[F_{\mathcal{B}},G]^\Sigma}
	\arrow["{U_l}"', from=1-1, to=3-1]
	\arrow["{U_l}", from=1-3, to=3-3]
	\arrow["{-\circ \iota}", from=1-1, to=1-3]
	\arrow["{-\circ \iota}"', from=3-1, to=3-3]
\end{tikzcd}\]
Therefore their right adjoints also commute, i.e.
\[\begin{tikzcd}
	{[F_{\sigma(\mathcal{B})},G]^\Sigma_l} && {[F_{\mathcal{B}},G]^\Sigma_l} \\
	\\
	{[F_{\sigma(\mathcal{B})},G]^\Sigma} && {[F_{\mathcal{B}},G]^\Sigma}
	\arrow["{R_l}", from=3-1, to=1-1]
	\arrow["{R_l}"', from=3-3, to=1-3]
	\arrow["{\text{Ran}^\text{lax}_\iota}"', from=1-3, to=1-1]
	\arrow["{\text{Ran}^\text{gen}_\iota}", from=3-3, to=3-1]
\end{tikzcd}\]
Therefore we have, \[\mathrm{Ran}^\mathrm{lax}_\iota\lambda = R_l(\mathrm{Ran}^\mathrm{gen}_\iota \lambda).\] Let $A$ be a countable set and let $x:=(X_a)_{a\in A}\in F_{\sigma(\mathcal{B})}A$. By Proposition \ref{Rl} and Lemma \ref{generaltransformations} together with the fact that $\mathrm{Ran}_{\iota_\textbf{1}}\lambda_{\textbf{1}}\circ \iota_\textbf{1} = \lambda_\textbf{1}$, since $\iota_\textbf{1}$ is full and faithful, we find that 
\begin{align*}
    [\mathrm{Ran}^\mathrm{lax}_\iota\lambda]_A(x)_a &= \inf\left\{\sum_{n=1}^\infty (\mathrm{Ran}^\mathrm{gen}_\iota\lambda)_\textbf{1}(A_n)\mid X_a\subseteq \bigcupdot_{n=1}^\infty A_n, A_k\in \sigma(\mathcal{B}) \right\}\\ 
    & =  \inf\left\{\sum_{n=1}^\infty \mathrm{Ran}_{\iota_\textbf{1}}\lambda_\textbf{1}(A_n)\mid X_a\subseteq \bigcupdot_{n=1}^\infty A_n, A_k\in \sigma(\mathcal{B})\right\}\\
    & \leq \inf\left\{\sum_{n=1}^\infty \mathrm{Ran}_{\iota_\textbf{1}}\lambda_\textbf{1}(B_n)\mid X_a\subseteq \bigcupdot_{n=1}^\infty B_n, \bigcupdot_{n=1}^\infty B_n\in \mathcal{B} \text{ and }B_k\in \mathcal{B}\right\}\\
    & = \inf\left\{\sum_{n=1}^\infty\lambda_\textbf{1}(B_n)\mid X_a\subseteq \bigcupdot_{n=1}^\infty B_n, \bigcupdot_{n=1}^\infty B_n\in \mathcal{B} \text{ and }B_k\in \mathcal{B}\right\}. 
\end{align*} 

We will now show that the inequality in the third line is in fact an equality.

First, it is easy to see that for $E\in \sigma(\mathcal{B})$ we have that  $$\mathrm{Ran}_{\iota_\textbf{1}}\lambda_\textbf{1}(E)=\inf\{\lambda_\textbf{1}(B)\mid E\subseteq B, B\in \mathcal{B}\},$$ as this is a usual right Kan extension of order-preserving maps between posets. Let $(A_n)_{n=1}^\infty$ be pairwise disjoint in $\sigma(\mathcal{B})$ such that $X_a\subseteq \bigcupdot_{n=1}^\infty A_n$. Let $\epsilon>0$. For every $n\geq 1$, there exists $A_n\subseteq B_n$ such that \[\mathrm{Ran}_{\iota_\textbf{1}}\lambda_\textbf{1}(A_n)+\frac{\epsilon}{2^n}\geq \lambda_\textbf{1}(B_n)\geq \lambda_\textbf{1}(\tilde{B}_n),\]
where $\tilde{B}_1:=B_1$ and $\tilde{B}_n:=B_n\setminus \tilde{B}_{n-1}$ for $n>1$. Summing over $n\geq 1$ gives \[\sum_{n=1}^\infty \mathrm{Ran}_{\iota_\textbf{1}}\lambda_{\textbf{1}}(A_n)+\epsilon\geq \sum_{n=1}^\infty \lambda_{\textbf{1}}(\tilde{B}_n) \geq \inf\left\{\sum_{n=1}^\infty \lambda_{\textbf{1}}(B_n)\mid X_a\subseteq \bigcupdot_{n=1}^\infty B_n, \bigcupdot_{n=1}^\infty B_n\in \mathcal{B}\right\}. \]Taking $\epsilon \to 0$ and taking the infimum over all such $(A_n)_{n=1}^\infty$, yields the claimed equality.

Using Proposition \ref{Rl}, it follows that \[\mathrm{Ran}^\mathrm{lax}_\iota\lambda\circ \iota = R_l(\lambda)=\lambda.\]
This shows that the extension is proper.
\end{proof}
\begin{thm}
Right extensions along $\iota$ of strict transformations exist and are proper.
\end{thm}
\begin{proof}  
Corollary \ref{C.3} states that $\overline{\mu\mid_\mathcal{B}}=\overline{\mu}\mid_\mathcal{B}$ for all outer measures $\mu$.
This shows that $\overline{\lambda}\circ \iota = \overline{\lambda\circ \iota}$ for all $\lambda \in [F_{\sigma(\mathcal{B})},G]_l^\Sigma$. Because $(F_\mathcal{B},G,\Sigma)$ and $(F_{\sigma(\mathcal{B})},G,\Sigma)$ satisfy the strictness condition, the result now follows from Corollary \ref{3.8} and Proposition \ref{7.2}.
\end{proof}
\begin{thm}[Carath\'edory]\label{7.4}
Let $\rho$ be a premeasure on a premeasurable space $(X,\mathcal{B})$. There exists a measure $\mu$ on $(X,\sigma(\mathcal{B}))$, such that $\mu(B)=\rho(B)$ for all $B\in \mathcal{B}$.  
\end{thm} 
\begin{proof}
Let $\mu$ be the measure corresponding to the right extension along $\iota$ of the strict transformation $\tau^\rho:F_\mathcal{B}\to G$. Since the extension is proper, it follows that $\mu(B)=\rho(B)$ for all $B\in \mathcal{B}$.
\end{proof}

\subsection{The left Carath\'eodory extension}
We start by showing that extensions of $\Sigma$-natural colax transformations along $\iota$ are objectwise and proper. This is an application of Theorem \ref{objectwisesigmacolax}.
 \begin{prop} \label{7.5}
Left extensions along $\iota$ of $\Sigma$-natural colax transformations exist and are objectwise and proper.   
 \end{prop}
\begin{proof}
 
We want to apply Theorem \ref{objectwisesigmacolax}. Therefore we will show that the three conditions of this result are satisfied.

Let $f:A\to B$ be a finite partial map. We know from Proposition \ref{proppreserves2} that $Gf$ preserves directed joins. This shows that $(G1)$ from Theorem \ref{objectwisesigmacolax} is satisfied. 

Consider a countable set $A$ and $x:=(X_a)_{a\in A}\in F_{\sigma(\mathcal{B})}A$ and $b^1:=(B^1_a)_{a\in A}$ and $b^2:=(B^2_a)_{a\in A}$ in $F_\mathcal{B}A$ such that $b^1\leq x$ and $b^2\leq x$. Then $(B^1_a\cup B^2_a)_{a\in A}$ are pairwise disjoint premeasurable subsets and clearly $B^1_a\cup B^2_a\subseteq X_a$ for all $a\in A$. Therefore $\iota_A\downarrow x$ is directed. This means that condition $(\iota 1)$ of Theorem \ref{objectwisesigmacolax} holds for $\kappa=\aleph_0$.

By Remark \ref{remark}, we can restrict $\Sigma$ to the collection of partial maps of the form $s_{a_0}:A\to \textbf{1}$, where $A$ is a countable set, $a_0$ is an element of $A$ and the domain of $s_{a_0}$ is $\{a_0\}$. 

For such a  partial map $s_{a_0}:A\to \textbf{1}$, we have that  $F_\mathcal{B}s_a:F_\mathcal{B}A\to F_\mathcal{B}\textbf{1}$ is defined by the assignment $$(B_a)_{a\in A}\mapsto B_{a_0}.$$

Now define $(F_\mathcal{B}s_{a_0})_*:F_\mathcal{B}\textbf{1}\to F_\mathcal{B}A$ by 

$$(F_\mathcal{B}s_{a_0})_*(B)_a:=\begin{cases} B \text{ if }a=a_0\\
\emptyset \text{ otherwise}\end{cases},$$ for all $B\in F_\mathcal{B}\textbf{1}$ and $a\in A$.

It is straightforward to verify that $(F_\mathcal{B}s_{a_0})_*$ is left adjoint to $F_\mathcal{B}s_{a_0}$ and these left adjoints satisfy condition $(\iota_2)$ in Theorem \ref{objectwisesigmacolax}. 

It follows now from Theorem \ref{objectwisesigmacolax} that left extensions along $\iota$ exist and are objectwise. Since $\iota_A$ is full and faithful for all countable sets $A$, we know by Lemma \ref{propernesslemma} that the extension is proper. 
\end{proof}

However, left extensions along $\iota$ of \emph{strict} transformations might not exist in general. In other words, the left Carath\'eodory extension of a premeasure does not always exist, as can be seen from the following counterexample, based on Example 4.20 in \cite{wise}. 

\begin{eg}
Let $\mathcal{B}$ be the algebra of subsets of $\mathbb{Q}$, generated by $\{(a,b]\cap \mathbb{Q}\mid a,b\in \mathbb{Q}\}$. Let $\rho:\mathcal{B}\to [0,\infty]$ be the premeasure that takes the value $\infty$ for all non-empty subsets in $\mathcal{B}$. Suppose that this premeasure has a left Carath\'eodory extension $\mu$. Note that since right Caratheodory extensions exist and are proper, we know by Proposition \ref{3.9} that the left Carath\'edoroy extension has to be proper. For $r>0$, there is a measure $\mu_r$ on $\mathbb{Q}$ that is defined by $\mu_r(\{q\})=r$ for all $q\in \mathbb{Q}$. We clearly have that $\mu_r(B)=\rho(B)$ for all $B\in \mathcal{B}$. By the universal property of left extensions it follows that $\mu\leq \mu_r$, i.e. for all $r>0$ and $q\in \mathbb{Q}$, we have \[\mu(\{q\})\leq r.\] This implies that $\mu=0$, which is clearly a contradiction. 
\end{eg}

\begin{rem}
It follows immediately from the $\pi$-$\lambda$ theorem (Theorem 1.19 in \cite{klenke}), that a \emph{finite} measure is determined by its values on a generating algebra of subsets. 
This means that if there exist a proper extension, it has to be unique. By the Carath\'eodory extension theorem (Theorem \ref{7.4}), we know that such an extension exist. 

It follows now that for \emph{finite} premeasures, the right and the left Carath\'edory extensions exist and are equal. 
\end{rem}
\newpage
\appendix 
 \section{Lax coends}\label{A}
 In this section we will give the definition of lax coends. More details about lax coends can be found in chapter 7 of \cite{loregian}.

Let $\mathcal{A}$ and $\mathcal{B}$ be poset-enriched categories and let $S:\mathcal{A}^\mathrm{op}\times \mathcal{A}\to \mathcal{B}$ be an enriched functor. 

\begin{defn}
A \textbf{lax cowedge} $\omega$ consists of \begin{itemize}
    \item an object $B$ in $\mathcal{B}$,
    \item a morphism $\omega_A: S(A,A)\to B$ for every object $A$ in $\mathcal{A}$ and 
    \item for every morphism $f:A_1\to A_2$, a 2-cell $\omega_f$ \[\begin{tikzcd}
	{S(A_1,A_2)} && {S(A_2,A_2)} \\
	\\
	{S(A_1,A_1)} && B
	\arrow["{\omega_{A_1}}"', from=3-1, to=3-3]
	\arrow["{\omega_{A_2}}", from=1-3, to=3-3]
	\arrow["{S(1_{A_1}\times f)}"', from=1-1, to=3-1]
	\arrow["{S(f\times 1_{A_2})}", from=1-1, to=1-3]
	\arrow["{\leq }"{description}, draw=none, from=3-1, to=1-3]
\end{tikzcd}.\]
\end{itemize}
\end{defn}

\begin{defn}
Let $\omega^1:=(B^1,(\omega^1_A)_A,(\omega^1_f)_f)$ and $\omega^2:=(B^2,(\omega^2_A)_A,(\omega^2_f)_f)$ be lax cowedges of $S$. A \textbf{morphism of cowedges} from $\omega^1$ to $\omega^2$ is a morphism $s:B_1\to B_2$ such that \[\begin{tikzcd}
	{S(A_1,A_2)} && {S(A_2,A_2)} && {S(A_1,A_2)} && {S(A_2,A_2)} \\
	&&& {=} \\
	{S(A_1,A_1)} && {B_1} && {S(A_1,A_1)} && {B_2} \\
	&&& {B_2}
	\arrow["{\omega_{A_1}}"{description}, from=3-1, to=3-3]
	\arrow["{\omega_{A_2}}"{description}, from=1-3, to=3-3]
	\arrow["{S(1_{A_1}\times f)}"', from=1-1, to=3-1]
	\arrow["{S(f\times 1_{A_2})}", from=1-1, to=1-3]
     \arrow["{\leq }"{description}, draw=none, from=3-1, to=1-3]
	\arrow["s"{description}, from=3-3, to=4-4]
	\arrow[""{name=0, anchor=center, inner sep=0}, "{\omega^2_{A_1}}"', from=3-1, to=4-4]
	\arrow[""{name=1, anchor=center, inner sep=0}, "{\omega^2_{A_2}}", from=1-3, to=4-4]
	\arrow["{S(f\times 1_{A_2})}", from=1-5, to=1-7]
	\arrow["{S(1_{A_1}\times f)}"', from=1-5, to=3-5]
	\arrow["{\omega^2_{A_1}}"{description}, from=3-5, to=3-7]
	\arrow["{\omega_{A_2}^2}"{description}, from=1-7, to=3-7]
     \arrow["{\leq }"{description}, draw=none, from=3-5, to=1-7]
	\arrow["{=}"{description}, draw=none, from=3-3, to=0]
	\arrow["{=}"{description}, draw=none, from=3-3, to=1]
\end{tikzcd}\]
for all morphism $f:A_1\to A_2$. 
\end{defn}

The category of lax cowedges of $S$ and their morphims is denoted by $\mathbf{Cowedge}_l(S)$. 

\begin{defn}
The lax coend of $S$ is an initial object in $\textbf{Cowedge}_l(S)$.
\end{defn}

If $(B,(\omega_A)_A,(\omega_f)_f)$ is the lax coend of $S$, i.e. the initial object in $\textbf{Cowedge}_l(S)$, then we also write \[B=:\ointclockwise^{A\in \ob\mathcal{A}}S(A,A).\]
Dually, we can define a category of colax cowedges of $S$, $\textbf{Cowedge}_c(S)$ and define colax coends as initial objects in this category. The notation we will use for the colax coend of $S$ is $\ointctrclockwise^{A\in \ob\mathcal{A}}S(A,A).$
 \section{Finite outer premeasures}\label{C}
An outer premeasure $\mu$ on $(X,\mathcal{B})$ is called \textbf{finite} if $\overline{\mu}(X)<\infty$.\footnote{Here we use the notation of Proposition \ref{6.10}} 
\begin{thm}\label{C.1}
Let $(X,\mathcal{B})$ be a premeasubrable space and let $\mu$ be an outer premeasure on $(X,\sigma(\mathcal{B}))$. Then $\mu$ is finite if and only if $\mu\mid_{\mathcal{B}}$ is finite. In this case for every $\epsilon>0$ and $A\in \sigma(\mathcal{B})$, there exists a $B\in \mathcal{B}$ such that $\mu(A\Delta B)<\epsilon.$ Also in this case, \[\overline{\mu}(X)=\overline{\mu\mid_{\mathcal{B}}}(X).\]
\end{thm}
\begin{proof}
If $\mu$ is finite then, $\overline{\mu\mid_{\mathcal{B}}}(X)\leq \overline{\mu}(X)<\infty$ and therefore $\mu\mid_{\mathcal{B}}$ is finite.

Suppose now that $\mu\mid_{\mathcal{B}}$ is finite. Define the set \[S:=\left\{A\in\sigma(\mathcal{B})\mid \forall \epsilon >0\ \exists B\in \mathcal{B}:\mu(A\Delta B)<\epsilon \right\}\]
If $A\in S$ and $\epsilon>0$, then there exists $B$ such that $\mu(A\Delta B)<\epsilon$. Therefore $\mu(A^c\Delta B)= \mu(A\Delta B^c)<\epsilon$. We conclude that $A^c\in S$ and therefore $S$ is closed under complements.

Let $\epsilon>0$. For $A_1,A_2\in S$, there exists $B_1,B_2\in \mathcal{B}$ such that $\mu(A_1\Delta B_1)<\frac{\epsilon}{2}$ and $\mu(A_2\Delta B_2)<\frac{\epsilon}{2}$. Since $(A_1\cup A_2)\Delta (B_1\cup B_2)\subseteq (A_1\Delta B_1)\cup (A_2\Delta B_2)$, it follows that $\mu((A_1\cup A_2)\Delta (B_1\cup B_2))\leq \epsilon$. Therefore $S$ is closed under finite unions. 

Let $\epsilon>0$ and let $(A_n)_{n=1}^\infty$ pairwise disjoint in $S$. For $N\geq 1$, there exist $(B_n)_{n=1}^N$ such that $\mu(A_n\Delta B_n)\leq \frac{\epsilon}{2N}$ for all $n\in \{1,\ldots, N\}$. We can choose $(B_n)_{n=1}^N$ such that they are pairwise disjoint.\footnote{Because $A_n\in \mathcal{S}$, there exists $\tilde{B}_n\in \mathcal{B}$ such that $\mu(A_n\Delta \tilde{B}_n)<\frac{\epsilon}{2N^2}$. Now define $B_1:=\tilde{B}_1$ and $B_n:= \tilde{B}_n\setminus \bigcup_{k=1}^{n-1}B_{k}$ for $n\geq 2$. Clearly $(B_n)_{n=1}^N$ is a pairwise disjoint collection of elements in $\mathcal{B}$. Moreover, since $B_n\subseteq \tilde{B}_n$ and $A_n\subseteq A_{n-1}^C$ for all $1\leq n\leq N$, we have \begin{align*}
  A_n\Delta B_n & = (A^c_n\cap B_n) \cup (A_n \cap B_n^c)\\
     &= (A^c_n\cap B_n)\cup (A_n \cap \tilde{B}_n^c) \cup (A_n\cap B_{n-1})\cup \ldots \cup(A_n\cap B_1)\\
    & \subseteq (A^c_n\cap \tilde{B}_n)\cup (A_n \cap \tilde{B}_n^c) \cup (A_{n-1}^c\cap \tilde{B}_{n-1})\cup \ldots \cup (A_1^c\cap \tilde{B}_1)\\
    & \subseteq (A_n\Delta \tilde{B}_n) \cup (A_{n-1}\Delta \tilde{B}_{n-1})\cup \ldots \cup (A_1\Delta \tilde{B}_1)\\
\end{align*}
Therefore $\mu(A_n\Delta B_n)\leq \frac{n\epsilon}{2N^2}\leq \frac{\epsilon}{2N}$.} It follows now that \[\sum_{n=1}^N\mu(A_n)\leq \sum_{n=1}^N(\mu(A_n\cap B_n)+\mu(A_n\Delta B_n))\leq \sum_{n=1}^N\left(\mu(B_n)+\frac{\epsilon}{N}\right) = \sum_{n=1}^N\mu(B_n)+\epsilon \leq \overline{\mu\mid_{\mathcal{B}}}(X)+\epsilon\]


By letting $N\to \infty$ and then letting $\epsilon \to 0$, we find for pairwise disjoint $(A_n)_{n=1}^\infty$ in $S$, that \begin{equation}\label{eq}
    \sum_{n=1}^\infty \mu(A_n)\leq \overline{\mu\mid_\mathcal{B}}(X) <\infty
\end{equation} 
It follows that there exists $N\geq 1$ such that $\sum_{n=N+1}^\infty\mu(A_n)<\frac{\epsilon}{2}$. Because \[\left(\bigcupdot_{n=1}^\infty A_n\right)\Delta\left(\bigcupdot_{n=1}^NB_n\right)\subseteq \left(\bigcup_{n=1}^NA_n\Delta B_n\right) \cup \bigcupdot_{n=N+1}^\infty A_n,\]
we have that \[\mu\left(\bigcupdot_{n=1}^\infty A_n\Delta\bigcupdot_{n=1}^NB_n\right)\leq \sum_{n=1}^N\mu(A_n\Delta B_n) + \sum_{n=N+1}^\infty\mu(A_n) \leq \epsilon.\]
We conclude that $S$ is closed under countable disjoint unions. This means that $S$ is a $\sigma$-algebra that contains $\mathcal{B}$ and therefore $S=\sigma(\mathcal{B})$. This already proves the second claim in the statement. By (\ref{eq}) it also follows that $\overline{\mu}(X)\leq \overline{\mu_\mathcal{B}}(X)$. This shows that $\mu$ is finite and the last claim in the statement.
\end{proof}
\begin{cor}\label{C.2}
Let $\mu$ be an outer measure on $(X,\sigma(\mathcal{B}))$, then $\overline{\mu}(X)=\overline{\mu\mid_{\mathcal{B}}}(X)$.
\end{cor}
\begin{proof}
If $\mu$ is finite, then this follows from Theorem \ref{C.1}. If $\mu$ is not finite then, $\mu\mid_{\mathcal{B}}$ is also not finite by Theorem \ref{C.1} and the result also follows.
\end{proof}
\begin{cor}\label{C.3}
Let $\mu$ be an outer measure on $(X,\sigma(\mathcal{B}))$, then $\overline{\mu\mid_{\mathcal{B}}}(B)=\overline{\mu}(B)$ for all $B\in \mathcal{B}$.
\end{cor}
\begin{proof}
Let $B\in \mathcal{B}$. Consider $\mathcal{B}':=\{C\in \mathcal{B}\mid C\subseteq B\}$. Note that $\mathcal{B}'$ is an algebra on the subset $B$, i.e. $\mathcal{B}'\subseteq \mathcal{P}(B)$ is a subalgebra. It can be checked that $\sigma(\mathcal{B}') = \sigma(\mathcal{B})'$.\footnote{Clearly $\mathcal
{B}'\subseteq \sigma(\mathcal{B})'$ and since $\sigma(\mathcal{B})'$ is a $\sigma$-algebra, we have $\sigma(\mathcal{B}')\subseteq \sigma(\mathcal{B})'$. For the other inclusion, consider \[S:=\{A\subseteq X\mid A\cap B\in \sigma(\mathcal{B'})\}.\] It is easy to see that $S$ is a $\sigma$-algebra containing $\mathcal{B}$. Therefore $\sigma(\mathcal{B})\subseteq S$. Thus for $A\in \sigma(\mathcal{B})'$, we have that $A\in S$, which means that $A=A\cap B\in \sigma(\mathcal{B}')$.} Let $\mu'$ be the outer measure on $(B,\sigma(\mathcal{B}'))$ defined by $\mu'(C):=\mu(C)$ for all $C\in \sigma(\mathcal{B}')$. Using the formula from Proposition \ref{6.10}, it is now easy to see that $\overline{\mu'}(B)=\overline{\mu}(B)$ and $\overline{\mu'\mid_{\mathcal{B}'}}(B)=\overline{\mu\mid_\mathcal{B}}(B)$. Using Corollary \ref{C.2}, we conclude that $\overline{\mu\mid_\mathcal{B}}(B)=\overline{\mu}(B)$ for all $B\in \mathcal{B}$.
\end{proof}
\bibliographystyle{abbrv}
\bibliography{ref.bib}
\end{document}